\documentclass[a4paper,11pt,reqno]{amsart}
\usepackage{amssymb,amsmath,mathrsfs,graphics,graphicx,mathptm,float}
\usepackage[T1]{fontenc}
\usepackage[english, francais]{babel}
\usepackage[all]{xy}
\newcommand{\Mt}{\ ^{t}\! M}

\newcommand{\PGL}{\mathrm{PGL}}
\newcommand{\im}{{\bf i}}

\newcommand{\Bir}{\mathrm{Bir}}
\newcommand{\Aut}{\mathrm{Aut}}

\newcommand{\C}{\mathbb{C}}
\newcommand{\p}{\mathbb{P}}
\newcommand{\z}{\mathbb{Z}}

\newcommand{\GL}{\mathrm{GL}}

\newcommand{\Pic}[1]{\mathrm{Pic}(#1)}

\theoremstyle{plain}
\newtheorem{thm}{Theorem}[section]
\newtheorem{pro}[thm]{Proposition}
\newtheorem{lem}[thm]{Lemma}
\newtheorem{cor}[thm]{Corollary}

\newtheorem{theoalph}{Theorem}

\theoremstyle{definition}

\newtheorem{eg}[thm]{Example}

\newtheorem{rem}[thm]{Remark}

\def\og{\leavevmode\raise.3ex\hbox{$\scriptscriptstyle\langle\!\langle$~}}
\def\fg{\leavevmode\raise.3ex\hbox{~$\!\scriptscriptstyle\,\rangle\!\rangle$}}

\usepackage[colorlinks, linktocpage, 
citecolor = blue, linkcolor = blue]{hyperref}

\addtocounter{section}{0}             
\numberwithin{equation}{section}       

\renewcommand{\b}{\mathfrak{b}}

\title{Degree growth of birational maps of the plane}
\thanks{Both authors were supported by the Swiss National Science Foundation grant no PP00P2\_128422 /1.}

\author{J\'er\'emy Blanc}
\address{Universit\"{a}t Basel, Mathematisches Institut, Rheinsprung $21$, CH-$4051$ Basel, Switzerland.}
\email{jeremy.blanc@unibas.ch}

\author{Julie D\'eserti}
\address{Universit\"{a}t Basel, Mathematisches Institut, Rheinsprung $21$, CH-$4051$ Basel, Switzerland.}
\email{julie.deserti@unibas.ch}
\medskip
\address{On leave from Institut de Math\'ematiques de Jussieu, Universit\'e Paris $7$, Projet G\'eom\'etrie et 
Dynamique, Site Chevaleret, Case $7012$, $75205$ Paris Cedex 13, France.}
\email{deserti@math.jussieu.fr}

\begin{document}
\selectlanguage{english}

\maketitle\begin{center}{\today}\end{center}

\begin{abstract}
This article studies the sequence of iterative degrees of a birational map of the plane. This sequence is known 
either to be bounded or to have a linear, quadratic or exponential growth.

The classification elements of infinite order with a bounded sequence of degrees is achieved, the case of elements 
of finite order being already known. The coefficients of the linear and quadratic growth are then described, and 
related to geometrical properties of the map. The dynamical number of base-points is also studied. 

Applications of our results are the description of embeddings of the Baumslag-Solitar groups and $\mathrm{GL}(2,
\mathbb{Q})$ into the Cremona group.

\noindent{\it 2010 Mathematics Subject Classification. --- 14E07, 37F10, 32H50  }
\end{abstract}

\tableofcontents

\section{Introduction}
A \emph{rational map} of the complex projective plane $\mathbb{P}^2=\mathbb{P}^2_{\C}$ into itself is a map of the following type
\begin{align*}
&\phi\colon\mathbb{P}^2\dashrightarrow\mathbb{P}^2, &&
(x:y:z)\dashrightarrow\big(\phi_0(x,y,z):\phi_1(x,y,z):\phi_2(x,y,z)\big),
\end{align*}
where the $\phi_i$'s are homogeneous polynomials of the same degree without common factor. The \emph{degree} $\deg \phi$ of~$\phi$ is by definition the degree of these polynomials. We will only consider \emph{birational} maps, which are rational maps having an inverse, and denote by $\Bir(\p^2)$ the group of such maps, classically called \emph{Cremona group}.

We are interested in the behaviour of the sequence $\left\{\deg \phi^k\right\}_{k\in \mathbb{N}}$. According to \cite{DiFa}, the sequence is either bounded or has a linear, quadratic or exponential growth. We will say that $\phi$ is

\begin{enumerate}
\item \emph{elliptic} if the growth is bounded;
\item a \emph{Jonqui\`eres twist} if the growth is linear;
\item an \emph{Halphen twist} if the growth is quadratic;
\item \emph{hyperbolic}  if the growth is exponential.
\end{enumerate}
This terminology is classical, and consistent with the natural action of $\Bir(\p^2)$ on an hyperbolic space of infinite dimension, where Jonqui\`eres and Halphen twists are parabolic (\cite[Theorem 3.6]{Ca}).\\

Recall that the first dynamical degree of  $\phi\in \Bir(\p^2)$ is  $\lambda(\phi)=\displaystyle\lim\limits_{k\to +\infty}(\deg \phi^k)^{1/k}\in \mathbb{R}$. This is an invariant of conjugation which allows to distinguish the first three cases (where $\lambda(\phi)=1$) from the last case (where $\lambda(\phi)>1$). There are plenty of articles on hyperbolic elements and the possible values for the algebraic integer $\lambda(\phi)$, we are here more interested in the growth of the first three cases.

 The nature of the growth is invariant under conjugation, and induces geometric properties on $\phi$, that we describe now. \\
 
 An element $\phi\in \Bir(\p^2)$ is elliptic if and only if it is conjugate to an automorphism $g\in \Aut(\mathrm{S})$ of a smooth projective rational surface $\mathrm{S}$ such that $g^n$ belongs to the connected component $\Aut^0(\mathrm{S})$ of $\Aut(\mathrm{S})$ for some $n>0$ (\emph{see} \cite[Theorem 0.2, Lemma 4.1]{DiFa}). Reading this description, one would expect to find examples where $g$ has infinite order and does not belong to the connected component. We will remove this possibility and refine the result of \cite{DiFa} in Section~\ref{Sec:DegBounded}, by showing that, up to conjugation, either $g$ has finite order or $\mathrm{S}=\p^2$. The complete classification of elements of finite order of $\Bir(\p^2)$ can be found in \cite{BlaC}; for elements of infinite order, one has (Proposition~\ref{Prop:ellipticinfiniteorder}):
 
\begin{theoalph}
If $\phi\in \Bir(\p^2)$ is elliptic of infinite order, then $\phi$ is conjugate to an automorphism of $\p^2$, which restricts to one of the following automorphisms on some open subset isomorphic to $\C^2$:
\begin{enumerate}
\item $(x,y)\mapsto (\alpha x,\beta y)$, where  $\alpha$, $\beta\in \C^{*}$, and where the kernel of the group homomorphism $\mathbb{Z}^2 \to \C^{*}$ given by $(i,j)\mapsto \alpha^i \beta^j$ is generated by $(k,0)$  for some $k\in \mathbb{Z}$;
 
\item $(x,y)\mapsto (\alpha x, y+1)$, where $\alpha\in \C^{*}$.
\end{enumerate}
\end{theoalph}
 
The end of Section~\ref{Sec:DegBounded} is devoted to the description of the conjugacy classes of such maps (Proposition \ref{Prop:ConjDiagAlmostDiag}) and their centralizers in the Cremona group (Lemmas~\ref{Lem:centraldiag} and \ref{Lem:centtransl}).\\

A birational map $\phi\in \Bir(\p^2)$ has a finite number $\b(\phi)$ of base-points (that may belong to $\p^2$ or correspond to infinitely near points). We will call the number $$\mu(\phi)=\lim\limits_{k\to +\infty} \frac{\b(\phi^k)}{k},$$ the \emph{dynamical number of base-points of $\phi$}. In Section~\ref{Sec:Basepointgrowth}, we study the sequence $\{{\b(\phi^k)}\}_{k\in \mathbb{N}}$ and deduce some properties on the number $\mu(\phi)$; let us state some of them. It is a non-negative integer invariant under conjugation; it also allows us to give a characterisation of birational maps conjugate to an automorphism of a smooth projective rational surface (Proposition \ref{Pro:RecCorMuInvariant}):

\begin{theoalph}
Let  $\mathrm{S}$ be a smooth projective surface; the birational map $\phi\in\Bir(\mathrm{S})$ is conjugate to an automorphism of a smooth projective surface if and only if $\mu(\phi)=0$.
\end{theoalph}

In the case where $\phi$ is a Jonqui\`eres twist, the number $\mu(\phi)$ determines the degree growth of~$\phi$. A Jonqui\`eres twist preserves an unique pencil of rational curves \cite[Theorem 0.2]{DiFa}. The sequence $\left\{\deg \phi^k\right\}_{k\in \mathbb{N}}$ grows as $\alpha k$ for some constant $\alpha\in \mathbb{R}$. The number $\alpha$ is not invariant under conjugation, but one can show that the minimal value is attained when the rational curves of the pencil are lines, and is an integer divided by~$2$. More precisely, one has (Proposition~\ref{pro:MinimalJonq}):

\begin{theoalph}\label{ThmJon}
Let $\phi\in \Bir(\p^2)$ be a Jonqui\`eres twist. 
\begin{enumerate}
\item The set $$\left.\left\{\lim\limits_{k\to +\infty} \frac{\deg(\psi\phi^k\psi^{-1})}{k}\ \right\vert\ \psi\in \Bir(\p^2)\right\}$$ admits a minimum, which is equal to $\frac{\mu(\phi)}{2}\in \frac{1}{2}\mathbb{N}$.

\item There exists an integer $a\in \mathbb{N}$ such that $$\lim\limits_{k\to +\infty} \frac{\deg(\phi^k)}{k}=\frac{\mu(\phi)}{2}\cdot a^2.$$ Moreover, $a=1$ if and only if $\phi$ preserves a pencil of lines.
\end{enumerate}
\end{theoalph}

\bigskip

The case of Halphen twists is similar. An Halphen twist preserves an unique pencil of elliptic curves (\emph{see} \cite{Gi}, \cite[Theorem 0.2]{DiFa}). Any such pencil can be sent by a birational map onto a pencil of curves of degree $3n$ with $9$ points of multiplicity $n$, called Halphen pencil. We obtain the following (Proposition~\ref{Pro:Halphen}):

\begin{theoalph}
Let $\phi\in \Bir(\p^2)$ be an Halphen twist. 

\begin{enumerate}
\item The set $$\left.\left\{\lim\limits_{k\to +\infty} \frac{\deg(\psi\phi^k\psi^{-1})}{k^2}\ \right\vert\ \psi\in \Bir(\p^2)\right\}$$ admits a minimum  $\kappa(\phi)\in \mathbb{Q}_{>0}$.

\item There exists an integer $a\ge 3$ such that $$\lim\limits_{k\to +\infty} \frac{\deg(\phi^k)}{k^2}=\kappa(\phi)\cdot\frac{a^2}{9}.$$

Moreover, $a=3$ if and only if $\phi$ preserves an Halphen pencil.
\end{enumerate}
\end{theoalph}

\bigskip

An application of our results is the description of birational maps whose two distinct iterates are conjugate, the non-existence of embeddings of Baumslag-Solitar groups into the Cremona group and the 
description of the embeddings of $\mathrm{GL}(2,\mathbb{Q})$  into the Cremona group (\S\,\ref{Sec:applications}):

\begin{theoalph}
Let $\phi$ denote a birational map of $\mathbb{P}^2$ of infinite order. Assume that $\phi^n$ and $\phi^m$
 are conjugate and  that $\vert m\vert\not=\vert n\vert$. Then, $\phi$ is conjugate to an automorphism of $\C^2$ of the form $(x,y)\mapsto (\alpha x,y+1)$, where $\alpha\in \C^{*}$ such that $\alpha^{m+n}=1$ or $\alpha^{m-n}=1$.
 
In particular, if $\phi$ is conjugate to $\phi^n$ for any positive integer $n$, then $\phi$ is conjugate to $(x,y)\mapsto (x,y+1)$.
 \end{theoalph}

\begin{theoalph}
If $\vert m\vert$, $\vert n\vert$, $1$ are distinct, there is no embedding of 
\begin{align*}
&\mathrm{BS}(m,n)=\langle r,\, s\,\vert\, rs^mr^{-1}=s^n\rangle, && m,\,n\in\mathbb{Z},\,mn\not=0
\end{align*}
into the Cremona group.
\end{theoalph}

\begin{theoalph}
Let $\rho\colon \GL(2,\mathbb{Q})\to \Bir(\p^2)$ be an embedding. Up to conjugation by an element of $\Bir(\p^2)$, there 
exists an odd integer $k$ and an homomorphism $\chi\colon \mathbb{Q}^{*}\to \mathbb{C}^{*}$ $($with respect to multiplication$)$ 
 such that
\begin{align*}
&\rho\left(\left[\begin{array}{cc}a & b\\ c & d\end{array}\right]\right)=\left(x\cdot 
\frac{\chi(ad-bc)}{(cy+d)^k},\frac{ay+b}{cy+d}\right), &&\forall \left[\begin{array}{cc}a & b\\ c & d\end{array}\right]\in 
\GL(2,\mathbb{Q}).
\end{align*}
\end{theoalph}

\begin{rem}
All these statements hold replacing $\mathbb{C}$ by an algebraically closed field of characteristic 
zero.
\end{rem}

\subsection*{Acknowledgements} The authors would like to thank Yves de Cornulier and Ivan Marin for interesting discussions 
on group theory. Thanks also to Charles Favre for questions on elliptic maps of infinite order and to Serge Cantat for his 
remarks and references.

\section{Elliptic maps}
\label{Sec:DegBounded}
\subsection{Classification of elliptic maps of infinite order}

If $\phi$ is a birational map of $\mathbb{P}^2$ such that $\left\{\deg\phi^k\right\}_{k\in\mathbb{N}}$ 
is bounded, it is conjugate to an automorphism $g$ of a smooth rational surface $\mathrm{S}$ such that the action of 
$g$ on $\mathrm{Pic}(\mathrm{S})$ is finite \cite[Lemma 4.1]{DiFa}. If $g$ has finite order, the possible conjugacy classes are 
completely 
classified in \cite{BlaC}. Here we deal with the case of elements of infinite order, classifying the possibilities 
and describing its centralizers in $\Bir(\p^2)$.
 
\begin{pro}\label{Prp:P2Fn}
Let $g$ be an automorphism of a smooth rational surface $\mathrm{S}$ which has infinite order but has a finite action 
on $\mathrm{Pic}(\mathrm{S})$. Then, there exists a birational morphism $\mathrm{S}\to\mathrm{X}$ where $\mathrm{X}$ 
is equal either to $\p^2$ or to an Hirzebruch surface $\mathbb{F}_n$ for $n\not=1$, which conjugates $g$ to an automorphism 
of $\mathrm{X}$.
\end{pro}

\begin{rem}
The proof of this result follows from a study of possible minimal pairs, which is similar to the one made in \cite{BlaGGD} for 
finite abelian subgroups of $\Bir(\p^2)$ (\emph{see} \cite[Lemmas 3.2, 6.1, 9.7]{BlaGGD}).
\end{rem}

\begin{proof}
Contracting the possible sets of disjoint $(-1)$-curves on $\mathrm{S}$ which are invariant by $g$, we can assume that the 
action of $g$ on $\mathrm{S}$ is minimal. The action of $g$ on $\mathrm{Pic}(\mathrm{S})$ being of finite order, the process 
corresponds to applying a $G$-Mori program, where $G$ is a finite group acting on $\mathrm{Pic}(\mathrm{S})$ (we only look at 
parts of the Picard group which are invariant). Then one of the following occurs (\cite{Man67, Isk79}):
\begin{enumerate}
\item
$\mathrm{Pic}(\mathrm{S})^g$ has rank $1$ and $\mathrm{S}$ is a del Pezzo surface;
\item
$\mathrm{Pic}(\mathrm{S})^g$ has rank $2$, and there exist a conic bundle $\pi\colon \mathrm{S}\to \p^1$ on $\mathrm{S}$, 
together with an automorphism $h$ of $\p^1$ such that $h\circ \pi=\pi\circ g$.
\end{enumerate}
We want to show that $\mathrm{S}$ is $\p^2$ or an Hirzebruch surface $\mathbb{F}_n$ for $n\not=1$, and exclude the other cases. 

In the case where $\mathrm{Pic}(\mathrm{S})^g$ has rank $1$, the fact that $g$ has infinite order but finite action on 
$\mathrm{Pic}(\mathrm{S})$ implies that the kernel of the group homomorphism $\Aut(\mathrm{S})\to \Aut(\mathrm{Pic}(\mathrm{S}))$ 
is infinite. So $\mathrm{S}$ is a del Pezzo surface of degree $(K_\mathrm{S})^2\ge 6$.  The surface cannot be $\mathbb{F}_1$ 
otherwise the exceptional section would be invariant. Similarly, it cannot be 
the unique del Pezzo surface of degree $7$, which has exactly three $(-1)$-curves, forming a chain (one touches the two others, 
which are disjoint), because the curve of the middle (and also the union of the two others) would be invariant. The only 
possibilities are thus $\p^2$, $\p^1\times\p^1=\mathbb{F}_0$, and the del Pezzo surface of degree $6$.

If $\mathrm{S}$ is the del Pezzo surface of degree $6$, any element $h\in\Aut(\mathrm{S})$ acting minimally on $\mathrm{S}$ 
has finite or\-der~\cite[Lemma 9.7]{BlaGGD}. Let us recall the simple argument. The del Pezzo surface of degree $6$ is 
isomorphic to 
$$\mathrm{S}=\left\{(x:y:z),(u:v:w)\in \p^2\times\p^2\, \vert\, xu=yv=zw\right\}.$$
The projections $\pi_1$, $\pi_2$ on each factor are birational morphisms contracting three $(-1)$-curves on $p_1=(1:~0:~0)$,
$p_2=(0:1:0)$ and $p_3=(0:0:1)$. 
The group $\mathrm{Pic}(\mathrm{S})$ is generated by the six $(-1)$-curves of $\mathrm{S}$, which are $E_{i}=(\pi_1)^{-1}(p_i)$ 
and $F_{i}=(\pi_2)^{-1}(p_i)$ for   $1\le i\le 3$ and form an hexagon. In fact, the action on the hexagon gives rise to an isomorphism 
$\Aut(\mathrm{S})\simeq (\mathbb{C}^{*})^2\rtimes (\mathrm{Sym}_3\rtimes \mathbb{Z}/2\mathbb{Z})=(\mathbb{C}^{*})^2\rtimes D_6$.
 The action of $g$ on $\mathrm{S}$ being minimal,~$g$  permutes cyclically the curves, and either $g$ or $g^{-1}$ acts as $(E_1
\to F_2\to E_3\to F_1\to E_2\to F_3)$. This implies that $g$ or $g^{-1}$ is equal to 
$$\big((x:y:z),(u:v:w)\big)\to \big((\alpha v:\beta w:u),(\beta y:\alpha z:\alpha\beta x)\big),$$
 for some $\alpha$, $\beta\in \C^{*}$, and has order $6$. \\

We can now assume the existence of an invariant conic bundle $\pi\colon \mathrm{S}\to \p^1$. If $\pi$ has no singular fibre, 
then~$\mathrm{S}$ is a Hirzebruch surface $\mathbb{F}_n$ and $n\not=1$ because of the minimality of the action. It remains to 
exclude the case where $\pi$ has at least one singular fibre. 
The minimality of the action on $\mathrm{S}$ implies that the two components of any singular fibre $F$ (which are two $(-1)$-curves) 
are exchanged by a power $g^k$ of $g$, and in particular that the whole singular fibre is invariant by $g^k$. Note that $k$ a priori 
depends on $F$.

We now prove that $g^k$ does not act trivially on the basis of the conic bundle. If $g^{k}$ acts trivially on the basis of the 
fibration, the automorphism $g^{2k}$ acts trivially on $\mathrm{Pic}(\mathrm{S})$; taking a birational morphism $\mathrm{S}\to 
\mathbb{F}_n$ which contracts a component in each singular fibre one conjugates $g^{2k}$ to an automorphism of $\mathbb{F}_n$, 
which fixes pointwise at least one section. The pull-back on $\mathrm{S}$ of this section intersects only one component in each 
singular fibre and its image by $g^{k}$ gives thus another section, also fixed by $g^{2k}$. The action of $g^k$ on a general 
fibre of $\pi$ exchanges the two points of the two sections and hence has order $2$: contradiction. 

The action of $g^k$ on the basis is non-trivial and fixes the point of $\p^1$ corresponding to the singular fibre; so the same 
holds for $g$ (recall that the fixed points of an element of $\Aut(\p^1)$ and any of its non-trivial powers are the same). This implies that $F$ is invariant by $g$, so its two components are exchanged by it (and thus~$k$ is odd).

 In particular, $g$ exchanges the two components of any singular fibre. This implies that the number of singular fibres of $\pi$ 
is at most $2$, so $\mathrm{S}$ is the blow-up of one or two points of an Hirzebruch surface. 
 
The fact that the two components of at least one singular fibre are exchanged gives a symmetry on the sections, that will help us 
to determine $\mathrm{S}$. Denote by $-m$ the minimal self-intersection of a section of~$\pi$ and let $s$ be one section which 
realises this minimum. Contracting the components in the singular fibres which do not intersect $s$, one has a birational morphism 
$\mathrm{S}\to \mathbb{F}_{n}$. The image of $s$ is a section with minimal self-intersection, so $m=n$. If $n=0$, then 
taking some section of $\mathbb{F}_0=\p^1\times\p^1$ of self-intersection~$0$ passing through at least one blown-up point, its strict transform on $\mathrm{S}$ would be a section of 
negative self-intersection, which contradicts the minimality
of $s^2$, so $m=n>0$. Let us denote by $s'$ the section $g(s)$, which also has self-intersection $-m$ on $\mathrm{S}$ but 
self-intersection $-m+r$ on $\mathbb{F}_m$, where $r$ is the number of singular fibres of $\pi$. Because any section of $\mathbb{F}_m$ 
distinct from the exceptional section has self-intersection  $\ge m$, we get $-m+r\ge m$, so $2\ge r\ge 2m$, which implies that $m=1$ 
and $r=2$. 

The surface $\mathrm{S}$ is thus the blow-up of two points on $\mathbb{F}_1$, not lying on the exceptional section and not on the same
 fibre, so is a del Pezzo surface of degree $6$. The fact that $g$ acts minimally on $\mathrm{S}$ is impossible, as we already 
observed.
\end{proof}

\begin{pro}\label{Prop:ellipticinfiniteorder}
Let $\phi$ be a birational map of $\mathbb{P}^2$ of infinite order, such that $\left\{\deg\phi^k\right\}_{k\in\mathbb{N}}$ is bounded. 

Then $\phi$ is conjugate to an automorphism of $\p^2$, which restricts to one of the following automorphisms on some open subset 
isomorphic to $\C^2$:
\begin{enumerate}
\item
 $(x,y)\mapsto (\alpha x,\beta y)$, where  $\alpha$, $\beta\in \C^{*}$, and where the kernel of the group homomorphism $\mathbb{Z}^2 
\to \C^{*}$ given by $(i,j)\mapsto \alpha^i \beta^j$ is generated by $(k,0)$  for some $k\in \mathbb{Z}$;
 \item
$(x,y)\mapsto (\alpha x, y+1)$, where $\alpha\in \C^{*}$.
\end{enumerate}
\end{pro}

\begin{proof}
According to Proposition~\ref{Prp:P2Fn} the map $\phi$ is conjugate to an automorphism of a minimal surface $\mathrm{S}$, equal to 
either $\p^2$ or an Hirzebruch surface. 

Suppose first that $\mathrm{S}=\p^2$. Looking at the Jordan normal form, any automorphism of $\p^2$ is conjugate to 
\begin{itemize}
\item either $(x:y:z)\mapsto (\alpha x:\beta y:z)$, 
\item or $(x:y:z)\mapsto (\alpha x: y+z:z)$,
\item or $(x:y:z)\mapsto (x+y:y+z:z)$.
\end{itemize}
This latter automorphism is conjugate to $(x:y:z)\mapsto (x:y+z:z)$ in $\Bir(\p^2)$ (for instance by $(x:y:z)\dasharrow (xz-
\frac{1}{2}y(y-z):yz:z^2)$, as already observed in  \cite[Example~1]{Bl}). It remains to study the case of diagonal automorphisms to 
show the assertion on the kernel stated in the proposition. As in the proof of \cite[Proposition~6]{Bl}, we associate to  each diagonal 
automorphism $\psi\colon(x:y:z)\mapsto(\alpha x: \beta y: z)$ the kernel $\Delta_\psi$ of the following homomorphism of groups:
\begin{align*}
& \delta_{\psi}\colon \mathbb{Z}^2 \to\C^{*}, && (i,j) \mapsto \alpha^i\beta^j.
\end{align*}
For any $M=\left[\begin{array}{cc} a & b \\ c & d\end{array}\right]\in \GL(2,\z)$, we denote by $M(\psi)$ the diagonal automorphism 
$$(x:y:z)\mapsto \big(\alpha^a\beta^b x:\alpha^c\beta^d y:z\big),$$ which is the conjugate of $\psi$ by the birational map $(x,y)\dasharrow 
\big(x^ay^b,x^cy^d\big)$ (viewed in the chart $z=1$). We can check that
$$\delta_{M(\psi)}=\delta_{\psi} \circ \Mt,$$
which implies that $\Delta_{M(\psi)}=\Mt^{-1}(\Delta_\psi)$.
We can always choose $M$ (by a result on Smith's normal form) such that $\Delta_{M(\phi)}$ is generated by $k_1 e_1$ and $k_1k_2 e_2$, 
where $e_1$, $e_2$ are the canonical basis vectors of $\mathbb{Z}^2$, and $k_1$, $k_2$ are non-negative integers, and replace $\phi$ 
with $M(\phi)$, which is conjugate to it. Since $\phi$ and $M(\phi)$ have infinite order, we see that $k_2=0$, and get the assertion 
on the kernel stated in the proposition.\\

If $\mathrm{S}=\mathbb{F}_0=\p^1\times \p^1$, we can reduce to the case of $\p^2$ by blowing-up a fixed point and contracting the
 strict transform of the members of the two rulings passing through the point. \\

Suppose now that $\mathrm{S}=\mathbb{F}_n$ for $n\ge 2$. If $g$ fixes a point of $\mathbb{F}_n$ which is not on the exceptional section, 
we can blow-up the point and contract the strict transform of the fibre to go to $\mathbb{F}_{n-1}$. We can thus assume that all points 
of $\mathbb{F}_n$ fixed by $g$ are on the exceptional section. The action of $g$ on the basis of the fibration is, up to conjugation,
$x\mapsto \alpha x$ or $x\mapsto x+1$  for some $\alpha\in \C^{*}$. Removing the fibre at infinity and the exceptional section, we get 
$\C^2$, where the action of $g$ is
\begin{itemize}
\item either $(x,y)\mapsto (\alpha x,\beta y+Q(x))$,
\item or $(x,y)\mapsto (x+1,\beta y+Q(x))$,
\end{itemize} where $\alpha$, $\beta\in \C^{*}$ and $Q$ is a polynomial of degree $\le n$.
The action on the fibre at infinity is obtained by conjugating by $(x,y)\dasharrow \left(\frac{1}{x},\frac{y}{x^n}\right)$.

In the first case, there is no fixed point on the fibre at infinity  (except the point on the exceptional section) if and only if 
$\beta=\alpha^n$ and $\deg Q=n$. There is no fixed point on $x=0$ if and only if $Q(0)\not=0$ and $\beta=1$. This implies that 
$\alpha$ is a primitive $k$-th rooth of unity, where $k$ is a divisor of $n$. Conjugating by $(x,y+\gamma x^d)$ we replace  $Q(x)$ 
with $Q(x)+\gamma (\alpha^d-1)x^d$, so we can assume that the coefficient of $x^d$ is trivial if $d$ is not a multiple of $k$, which 
means that $Q(x)=P(x^k)$ for some polynomial $P\in \C[x]$. In particular, $g$ is equal to~$(x,y)\mapsto (\xi x,y+P(x^k))$ and is 
conjugate to $(x,y)\mapsto \big(\xi x,y+1\big)$ by $(x,y)\dasharrow \left(x,\frac{y}{P(x^k)}\right)$.

In the second case, there is no fixed point on $\C^2$, and no point on the fibre at infinity  if and only if $\beta=1$ and $\deg Q=n$. 
Conjugating $g$ by $(x,y)\mapsto (x,y+ \gamma x^{n+1})$ (which corresponds to performing an elementary link $\mathbb{F}_{n}\dasharrow 
\mathbb{F}_{n+1}$ at the unique fixed-point and then coming back with an elementary link at a general point of the fibre at infinity), 
we get $$(x,y)\mapsto \big(x+1,y-\gamma x^{n+1}+Q(x)+\gamma (x+1)^{n+1}\big).$$ Choosing the right element $\gamma\in \C$, we can decrease the 
degree of $Q(x)$, and get $(x,y)\mapsto (x+1,y)$ by induction.
\end{proof}

\subsection{Conjugacy classes of elliptic maps of infinite order}

Following \cite{Bl}, we will call  elements of the form $(x,y)\mapsto (\alpha x,\beta y)$, resp. $(x,y)\mapsto (\alpha x,y+1)$ 
\emph{diagonal} automorphisms, resp. \emph{almost-diagonal} automorphisms of $\C^2$ (or $\p^2)$. The conjugacy classes in each 
family are given by the following:

\begin{pro}[\cite{Bl}, Theorem 1]\label{Prop:ConjDiagAlmostDiag}
\begin{enumerate}
\item
A diagonal automorphism and an almost-diagonal automorphism of $\C^2$ are never conjugate in $\Bir(\C^2)$.
\item
Two diagonal automorphisms $(x,y)\mapsto (\alpha x,\beta y)$ and $(x,y)\mapsto (\gamma x,\delta y)$ are conjugate in $\Bir(\C^2)$ if 
and only if there exists $\left[\begin{array}{cc} a& b\\ c&d\end{array}\right]\in \GL(2,\z)$ such that $(\alpha^a\beta^b,\alpha^c
\beta^d)=(\gamma,\delta)$.
\item

Two almost diagonal automorphisms $(x,y)\mapsto (\alpha x,y+1)$ and $(x,y)\mapsto (\gamma x,y+1)$ are conjugate in~$\Bir(\C^2)$ if 
and only if $\alpha=\gamma^{\pm 1}$.\end{enumerate}
\end{pro}

\begin{cor}\label{Cor:Lin}
Let $\phi\in \Bir(\p^2)$ be an elliptic map which has infinite order. If $\phi^m$ is conjugate to $\phi^n$ in~$\Bir(\p^2)$ for some 
$m$, $n\in \mathbb{Z}$, $\vert m\vert \not=\vert n\vert$, then $\phi$ is conjugate to an automorphism of $\C^2$ of the form  
$(x,y)\mapsto (\alpha x,y+1)$, where $\alpha\in \C^{*}$ such that $\alpha^{m+n}=1$ or $\alpha^{m-n}=1$.
\end{cor}

\begin{proof}
Note that $mn\not=0$ since $\phi$ has infinite order. Then $\phi$ is conjugate to one of the two cases of 
Proposition~\ref{Prop:ellipticinfiniteorder}.

First of all, assume that up to conjugation $\phi$ is $(x,y)\mapsto (\alpha x,\beta y)$ and that the kernel $\Delta_\phi$ of 
the group homomorphism $\mathbb{Z}^2 \to \C^{*}$ given by $(i,j)\mapsto \alpha^i \beta^j$ is generated by $(k,0)$ for some $k 
\in \mathbb{Z}$.
Since $\phi^m$ and $\phi^n$ are conjugate there exists a matrix $$N=\left[\begin{array}{cc} a & b \\ c & d\end{array}\right]\in
\mathrm{GL}(2,\mathbb{Z})$$ such that $\Big((\alpha^m)^a(\beta^m)^b,(\alpha^m)^c(\beta^m)^d\Big)=(\alpha^n,\beta^n)$ 
(Proposition~\ref{Prop:ConjDiagAlmostDiag}). This means that $$\Big(\alpha^{ma-n}\beta^{mb},\alpha^{mc}\beta^{md-n}\Big)=(1,1),$$ 
so $(ma-n,mb)$, $(mc,md-n)$ belong to $\Delta_\phi$. In particular $mb=md-n=0$, which implies that $b=0$, so $ad=\pm 1$, which is impossible 
since $m\not=\pm n$. \\

Assume now that $\phi$ is  conjugate to $(x,y)\mapsto (\alpha x, y+1)$ for some~$\alpha$ in~$\C^*$.
The fact that $\phi^m$ and $\phi^n$ are conjugate implies that $\alpha^{m+n}=1$ or $\alpha^{m-n}=1$ 
(Proposition~\ref{Prop:ConjDiagAlmostDiag}).
\end{proof}

\subsection{Centralisers of elliptic maps of infinite order}

If $\phi$ is a birational map of $\mathbb{P}^2$, we will denote by $\mathrm{C}(\phi)$ the centraliser of $\phi$ in $\mathrm{Bir}
(\mathbb{P}^2)$: $$\mathrm{C}(\phi)=\left\{\psi\in\mathrm{Bir}(\mathbb{P}^2)\,\big\vert\,\phi\psi=\psi\phi\right\}.$$

In the sequel, we describe the centralisers of elliptic maps of infinite order of $\Bir(\p^2)$. The results are groups which contain 
the centralisers of some elements of $\PGL(2,\C)$. We recall the following  result, whose proof is an easy exercise. Recall that 
$\PGL(2,\C)$ is the group of automorphisms of $\p^1$, or equivalently the group of M\"obius transformations $x\dasharrow 
\frac{ax+b}{cx+d}$.

\begin{lem}
For any $\alpha\in \C^{*}$, we have $$\left\{\eta\in \PGL(2,\C)\, \Big\vert\, \eta(\alpha x)=\alpha\eta( x)\right\}=
\left\{\begin{array}{lcl}
\PGL(2,\C)&\mbox{ if }&\alpha=1\\
\{x\dasharrow \gamma x^{\pm 1}\, \vert\, \gamma \in \C^{*}\}& \mbox{ if }& \alpha=-1\\
\{x\mapsto \gamma x\, \vert\, \gamma \in \C^{*}\}& \mbox{ if }& \alpha\not=\pm 1\end{array}\right.$$
\end{lem}

\begin{lem}\label{Lem:centraldiag}
Let us consider $\phi\colon(x,y)\mapsto (\alpha x,\beta y)$ where $\alpha$, $\beta$ are in $\mathbb{C}^*$, and where the kernel of 
the group homomorphism $\mathbb{Z}^2 \to \C^{*}$ given by $(i,j)\mapsto \alpha^i \beta^j$ is generated by $(k,0)$  for some 
$k\in \mathbb{Z}$. Then the centraliser of~$\phi$ in $\Bir(\p^2)$ is 
$$\mathrm{C}(\phi)=\left\{(x,y)\dashrightarrow (\eta (x),y R(x^k))\, \Big\vert\, R\in \C(x), \eta \in \PGL(2,\C), \eta(\alpha x)
=\alpha\eta(x)\right\}.$$
\end{lem}

\begin{proof}
Let $\psi\colon(x,y)\dashrightarrow(\psi_1(x,y),\psi_2(x,y))$ be an element of $\mathrm{C}(\phi)$. The fact that $\psi$ commutes with $\phi$ is equivalent to 
\begin{align*}
& (\star)\,\,\,\psi_1(\alpha x,\beta y)=\alpha\psi_1(x,y) && \text{and} && (\diamond)\,\,\,\psi_2(\alpha x,\beta y)=\beta\psi_2(x,y).
\end{align*}
Writing $\psi_i=\frac{P_i}{Q_i}$ for $i=1,2$, where $P_i$, $Q_i$ are polynomials without common factors, we see that $P_1,P_2,Q_1,Q_2$ 
are eigenvectors of the linear automorphism $\phi^{*}$ of the $\C$-vector space $\C[x,y]$ given by $\phi^{*}\colon f(x,y)\mapsto 
f(\alpha x,\beta y)$. This means that each of the $P_i$, $Q_i$ is a product of a monomial in $x,y$ with an element of $\C[x^k]$. 
Using $(\star)$ and~$(\diamond)$, we get the existence of $R_1$, $R_2\in \C(x)$ such that 
\begin{align*}
&\psi_1(x,y)=x R_1(x^k),&&\psi_2(x,y)=yR_2(x^k).
\end{align*}
The fact that $\psi$ is birational implies that $\psi_1(x,y)$ is an element $\eta(x)\in\PGL(2,\C)$; it satisfies $\eta(\alpha x)=\alpha \eta(x)$ because of $(\star)$.
\end{proof}

\begin{lem}\label{Lem:centtransl}
Let us consider $\phi\colon(x,y)\mapsto (\alpha x,y+\beta)$ where $\alpha,\beta\in\mathbb{C}^*$.
The centraliser of $\phi$ in $\Bir(\mathbb{P}^2)$ is equal~to $$\mathrm{C}(\phi)=\left\{(x,y)\dashrightarrow(\eta(x),y+R(x))\,\big\vert\,\eta\in \PGL(2,\C), \eta(\alpha x)=\alpha \eta(x), R\in \C(x), R(\alpha x)=R(x)\right\}.$$
\end{lem}

\begin{proof}Conjugating by $(x,y)\mapsto (x,\beta y)$, we can assume that $\beta=1$.

Let $\psi\colon(x,y)\dashrightarrow(\psi_1(x,y),\psi_2(x,y))$ be a birational map of $\mathbb{P}^2$ which commutes with $\phi$. One has 
\begin{align*}
&(\star)\,\,\, \psi_1(\alpha x,y+1)=\alpha\psi_1(x,y) && \text{and} &&(\diamond)\,\,\,\psi_2(\alpha x,y+1)=\psi_2(x,y)+1.
\end{align*}

\smallskip

Equality $(\star)$ implies that $\psi_1$ only depends on $x$ (\emph{see} \cite[Lemma~2]{Bl}). Therefore $\psi_1$ is an element of~$\PGL(2,\C)$ which commutes with $x\mapsto \alpha x$. 

Equality $(\diamond)$ implies that 
\begin{align*}
&\frac{\partial\psi_2}{\partial y}(\alpha x,y+1)=\frac{\partial\psi_2}{\partial y}(x,y)&& \text{and} &&\frac{\partial\psi_2}{\partial x}
(\alpha x,y+1)=\alpha^{-1}\frac{\partial\psi_2}{\partial x}(x,y),
\end{align*}
which again means that $\frac{\partial\psi_2}{\partial x}(x,y)$ and $\frac{\partial\psi_2}{\partial y}(x,y)$ only depend on $x$. The second component of $\psi$ can thus be written $ay+B(x)$, where $a\in \C^{*}$, $B\in \C(x)$. Replacing this form in $(\diamond)$, we get $$B(\alpha x)=B(x)+1-a,$$ which implies that $\frac{\partial B}{\partial x}(\alpha x)=\alpha^{-1}\frac{\partial B}{\partial x}(x),$ and thus that $x\frac{\partial B}{\partial x}(x)$ is invariant under $x\mapsto \alpha x$. 

If  $\alpha$ is not a root of unity, this means that $\frac{\partial B}{\partial x}=c/x$ for some $c\in \C$; since $B$ is a rational function, one gets~$c=0$ and  $B$ is a constant (or equivalently an element such that $B(\alpha x)=B(x)$). It implies moreover $a=1$ 
and we are done.
 
 If $\alpha$ is a primitive $k$-th root of unity, the fact that $\psi\colon (x,y)\dasharrow (\eta(x),a y+B(x))$ commutes with $$\phi^k\colon (x,y)\dasharrow(x,y+k)$$ yields $a(y+k)+B(x)=ay+B(x)+k$, so $a=1$. We again get $B(\alpha x)=B(x)$.\end{proof}

\section{On the growth of the number of base-points}\label{Sec:Basepointgrowth}
If $\mathrm{S}$ is a projective smooth surface, any element $\phi\in \Bir(\mathrm{S})$ has a finite number of base-points, which can belong to $\mathrm{S}$ or be infinitely near. We denote by $\b(\phi)$ the number of such points. We will call the number $$\mu(\phi)=\lim\limits_{k\to +\infty} \frac{\b(\phi^k)}{k},$$ the \emph{dynamical number of base-points of $\phi$}. Since $\b(\phi\psi)\le \b(\phi)+\b(\psi)$ for any $\phi$, $\psi\in 
\Bir(\mathrm{S})$, we see that~$\mu(\phi)$ is a non-negative real number. Moreover, $\b(\phi^{-1})$ and $\b(\phi)$ being always equal, we get $\mu(\phi^k)=\vert k\cdot \mu(\phi)\vert$ for any $k\in \mathbb{Z}$.

In this section, we precise the properties of this number, and will in particular see that it is an integer.\\

If $\phi\in \Bir(\mathrm{S})$ is a birational map, we will say that a (possibly infinitely near) base-point $p$ of $\phi$ is a {\it persistent base-point} if there exists an integer $N$ such that $p$ is a base-point of $\phi^k$ for any $k\geq N$ but is not a base-point of $\phi^{-k}$ for any $k\ge N$. \\

If $p$ is a point of $\mathrm{S}$ or a point infinitely near, which is not a base-point of $\phi\in \Bir(\mathrm{S})$, we define 
a point~$\phi^\bullet (p)$, which will also be a point of $\mathrm{S}$ or a point infinitely near. For this, take a minimal 
resolution 
$$
\xymatrix{& \mathrm{Z}\ar[rd]^{\pi_2}\ar[ld]_{\pi_1}&\\
\mathrm{S}\ar@{-->}[rr]_{\phi}&&\mathrm{S},
}$$
where $\pi_1$, $\pi_2$ are sequences of blow-ups. Because $p$ is not a base-point of $\phi$, it corresponds, via $\pi_1$, to a point 
of $\mathrm{Z}$ or infinitely near. Using $\pi_2$, we view this point on $\mathrm{S}$, again maybe infinitely near, and call it 
$\phi^\bullet(p)$.

\begin{rem}
If $p$ is not a base-point of $\phi\in \Bir(\mathrm{S})$ and $\phi(p)$ is not a base-point of $\psi\in \Bir(\mathrm{S})$, we have 
$(\psi\phi)^{\bullet}(p)=\psi^\bullet(\phi^\bullet(p))$. If $p$ is a general point of $\mathrm{S}$, then $\phi^{\bullet}(p)=
\phi(p)\in S$.
\end{rem}

\begin{eg} If  $\mathrm{S}=\p^2$, $p=(1:0:0)$ and $\phi$ is the birational map $(x:y:z)\dasharrow (yz+x^2:xz:z^2),$
the point $\phi^{\bullet}(p)$ is not equal to $p=\phi(p)$, but is infinitely near to it.
\end{eg}

Using this definition, we put an equivalence class on the set of points that belong to $\mathrm{S}$ or are infinitely near, by saying 
that $p$ is {\it equivalent} to $q$ if there exists an integer $k$ such that $(\phi^k)^{\bullet}(p)=q$ (this implies that $p$ is not 
a base-point of~$\phi^k$ and that $q$ is not a base-point of $\phi^{-k}$). The set of equivalence classes is the generalisation of 
the notion of set of orbits for birational maps.

\begin{pro}\label{Pro:MuPersists}
Let $\phi$ be a birational map of a smooth projective surface $\mathrm{S}$. Denote by $\nu$ the number of equivalence classes of 
persistent base-points of $\phi$. Then, the set
$$\left\{\b(\phi^k)- {\nu}k\ \Big\vert\ k\ge 0\right\}\subset \mathbb{Z}$$ is bounded. 
 
In particular, $\mu(\phi)$ is an integer, equal to $\nu$.
\end{pro}

\begin{proof}
Let us say that a base-point $q$ is periodic if $(\phi^k)^\bullet(q)=q$ for some $k\not=0$, or if $q$ is a base-point of $\phi^k$ for 
any $k\in\mathbb{Z}\setminus\{0\}$ (which implies that $(\phi^k)^\bullet(q)$ is never defined for $k\not=0$). Let us denote by $\mathcal{P}$ the set of periodic 
base-points of $\phi$ and by $\widehat{\mathcal{P}}$ the finite set of points equivalent to a point of $\mathcal{P}$.

The number of base-points of $\phi$ and $\phi^{-1}$ being finite, there exists an integer $N$ such that for any non-periodic base-point 
$p$ and for any $j$, $j'\ge N$, $p$ is a base-point of $\phi^j$ (respectively of  $\phi^{-j}$) if and only if $p$ is a base-point of 
$\phi^{j'}$ (respectively of $\phi^{-j'}$).
 
We decompose the set of non-periodic base-points of $\phi$ into four sets:
$$\begin{array}{ll}
\mathcal{B}_{++}=& \left\{p\ \vert\ p \mbox{ is a base-point of } \phi^j,  \mbox{and  is a base-point of } \phi^{-j} \mbox{ for }
j\ge N\right\},\\
\mathcal{B}_{+-}=& \left\{p\ \vert\ p \mbox{ is a base-point of } \phi^j,  \mbox{but  is not a base-point of } \phi^{-j} \mbox{ for }
j\ge N\right\},\\
\mathcal{B}_{-+}=& \left\{p\ \vert\ p \mbox{ is not a base-point of } \phi^j,  \mbox{but  is a base-point of } \phi^{-j} \mbox{ for }
j\ge N\right\},\\
\mathcal{B}_{--}=& \left\{p\ \vert\ p \mbox{ is not a base-point of } \phi^j,  \mbox{and  is not a base-point of } \phi^{-j} \mbox{ for }
j\ge N\right\}.
\end{array}$$
Note that $\mathcal{B}_{+-}$ is the set of persistent base-points of $\phi$ and that $\mathcal{B}_{-+}$ is the set of persistent 
base-point of $\phi^{-1}$. 
This decomposes the set of base-points of $\phi$ into five disjoint sets. Two base-points $p$, $p'$ of $\phi$ which are equivalent 
belong to the same set.

We fix an integer $k\ge 2N$ and compute the number of base-points of $\phi^k$.  Any such base-point being equivalent to a base-point of 
$\phi$, we take a base-point $p$ of $\phi$, and count the number $m_{p,k}$ of base-points of $\phi^k$ which are equivalent to $p$. 

If $p$ belongs to $\mathcal{P}$, the number of points equivalent to $p$ is less than $\#(\widehat{\mathcal{P}})$, so $m_{p,k}\le 
\#(\widehat{\mathcal{P}})$.

If $p$ is not in $\mathcal{P}$, any point equivalent to $p$ is equal to $(\phi^{i})^\bullet(p)$ for some $i$, and all are distinct, so 
we have
$$m_{p,k}=\# I_{p,k}, \mbox{ where }  I_{p,k}=\left\{i\in \mathbb{Z}\ \Big\vert \  p \mbox{ is not a base-point of }\phi^i, \mbox{ but }
 p \mbox{ is a base-point of }\phi^{i+k}\right\}.$$
 
If $p\in \mathcal{B}_{++}$, since $p$ is not a base-point of $\phi^i$ one has $-N< i< N$, thus $m_{p,k}< 2N$.

A point $p\in \mathcal{B}_{--}$ is a base-point of $\phi^{i+k}$ hence $-N<i+k<N$ and $m_{p,k}<2N$.

If $p\in \mathcal{B}_{-+}$, the fact that $p$ is not a base-point of $\phi^i$ implies that $-N<i$ and the fact that $p$ is a base-point 
of $\phi^{i+k}$ implies that $i+k\le N$. With these two inequalities one has $-N<i\le N-k$. Since $k>2N$, we get $m_{p,k}=0$.

If $p\in \mathcal{B}_{+-}$, the fact that $p$ is not a base-point of $\phi^i$ implies that $i<N$ and the fact that $p$ is a base-point 
of $\phi^{i+k}$ implies that $-N<i+k$. This yields $-N-k<i<N$, so $m_{p,k}\le 2N+k$. Conversely, if $i\le -N$ and $i+k\ge N$, $p$ is not 
a base-point of $\phi^i$, but  $p$ is a base-point of $\phi^{i+k}$ (or equivalently $i\in I_{p,k}$), so~$m_{p,k}\ge \#[N-k,-N]=k-2N+1$. 
The two conditions together imply that $m_{p,k}-k\in [-2N,2N]$.

Recall that $\mathcal{B}_{+-}$ is the set of persistent base-points of $\phi$. The above counting explains that the number of base-points of $\phi^k$, for $k$ big, behaves like $\nu k$, where $\nu$ is the set of equivalence classes of of persistent base-points of $\phi$. More precisely, there exist two integers $c$, $d$ which do not depend on $k$, such that the total number of base-points of~$\phi^k$ is  between 
$\nu k+c$ and $\nu k+d$, for any $k\ge 2N$. Recalling that 
$\mu(\phi)=\lim\limits_{k\to +\infty} \frac{\b(\phi^k)}{k}$, where $\b(\phi^k)$ is the number of base-points of $\phi^k$, we obtain $\mu(\phi)=\nu$.
\end{proof}

 \begin{cor}\label{Cor:MuInvariant}
 The  dynamical number of base-points  is an invariant of conjugation: if  $\theta\colon \mathrm{S}\dasharrow \mathrm{Z}$ is a birational map between 
smooth projective surfaces and $\phi\in\Bir(\mathrm{S})$, then
 $$\mu (\phi)=\mu(\theta\phi\theta^{-1}).$$

In particular, if $\phi$ is conjugate to an automorphism of a smooth projective surface, then $\mu(\phi)=0$.
 \end{cor}                                          
 
\begin{proof}
The map $\theta$ factorises into the blow-up of a finite number of base-points followed by the contraction of a finite number of 
curves. The number of equivalence classes of persistent base-points of $\phi$ and $\theta\phi\theta^{-1}$ is thus the same, and we 
get the result from Proposition~\ref{Pro:MuPersists}.

It is thus clear that $\mu(\phi)=0$ if $\phi$ is conjugate to an automorphism of a smooth projective surface.
\end{proof}
 
 \begin{pro}\label{Pro:RecCorMuInvariant}
 Let  $\mathrm{S}$ be a smooth projective surface, and let $\phi\in\Bir(\mathrm{S})$. The following conditions are equivalent:
 \begin{enumerate}
 \item
   $\mu(\phi)=0$;
   \item
   $\phi$ is conjugate to an automorphism of a smooth projective surface.
   \end{enumerate}
 \end{pro}
 
 \begin{proof}
 Corollary~\ref{Cor:MuInvariant} yields $(2)\Rightarrow (1)$. It suffices then to show $(1)\Rightarrow (2)$.
 
 Denote by $K$ the set of points, that belong to $\mathrm{S}$ as proper or infinitely near points, which are base-points of $\phi^i$ and $\phi^{-j}$, for some $i$, $j>0$. Let us prove that $K$ is a finite set. Choosing the smallest possible $i,j$ associated to $p\in K$, the  point $p$ is equivalent to $(\phi^{i-1})^\bullet(p)$, which is a base-point of $\phi$, and the same holds replacing $p$ with $(\phi^{k})^{\bullet}(p)$, for $-j< k<i$. By this way, we associate a finite number of points of $K$  to each base-point of~$\phi$. This shows that $K$ is finite. Observing that any point of $K$ is either a proper point of $\mathrm{S}$ or is infinitely near, we can blow-up the set $K$, and obtain a birational morphism $\pi\colon\widehat{\mathrm{S}}\to \mathrm{S}$.
 
 By construction, the birational map $\widehat{\phi}\in \Bir(\widehat{\mathrm{S}})$ given by $\pi^{-1}\phi\pi$ has less base-points than $\phi$, and satisfies that no point of $\widehat{\mathrm{S}}$ is a base-point of $\widehat{\phi}^i$ and $\widehat{\phi}^{-j}$, for some $i$, $j>0$ (in the notation of \cite{DiFa}, the map $\widehat{\phi}$ is now \emph{algebraically stable}). If a $(-1)$-curve of $\widehat{\mathrm{S}}$ (a smooth curve isomorphic to $\p^1$ and having self-intersection $(-1)$) is contracted by $\widehat\phi$, we contract it, via a birational morphism $\eta_1\colon \widehat{\mathrm{S}}\to \widehat{\mathrm{S}}_1$. The map $\widehat{\phi}_1=\eta_1\circ\widehat{\phi}\circ (\eta_1)^{-1}$ is again algebraically stable, and we continue the process if a $(-1)$-curve of $\widehat{\mathrm{S}}_1$ is contracted by $\widehat\phi_1$. At the end, we obtain a birational morphism $\eta=\eta_k\circ \dots \circ \eta_1\colon \widehat{\mathrm{S}}\to \widehat{\mathrm{S}}_k=\widetilde{\mathrm{S}}$ which conjugates $\widehat\phi$ to $\phi_k=\widetilde\phi$, and such that no $(-1)$-curve of $\widetilde{\mathrm{S}}$ is contracted by $\widetilde\phi$. 
 
 It remains to show that $\widetilde\phi$ is an automorphism of $\widetilde{\mathrm{S}}$. Assuming the converse, we will deduce a contradiction. We write $\tau_1\colon \mathrm{X}_1\to \widetilde{\mathrm{S}}$ the blow-up of the base-points of $\widetilde\phi$, and write $\chi_1\colon \mathrm{X}_1\to \widetilde{\mathrm{S}}$ the morphism $\chi_1=\widetilde\phi\tau_1$, which is the blow-up of the base-points of $\widetilde\phi^{-1}$. Denote by $C_1\subset \mathrm{X}_1$ a $(-1)$-curve contracted by $\chi_1$ onto a base-point~$p$ of $\widetilde\phi^{-1}$ (such a curve exists for each base-point). Because $p$ is a base-point of $\widetilde\phi^{-1}$, it is not a base-point of $\widetilde\phi^{r}$ for any $r>0$. The map $\widetilde\phi$ has no persistent base-point, because $\mu(\phi)=\mu(\widetilde\phi)=0$ (Proposition~\ref{Pro:MuPersists} and Corollary~\ref{Cor:MuInvariant}). In consequence, there exists $k>1$ such that $p$ is not a base-point of $\widetilde\phi^{-k}$.

We write $\chi_2\colon \mathrm{X}_2\to \mathrm{X}_1$ the blow-up of the base-points of $\widetilde\phi^{1-k}\tau_1$, and denote by $\tau_2\colon \mathrm{X}_2\to \widetilde{\mathrm{S}}$ the morphism $\tau_2=\widetilde\phi^{1-k}\tau_1\chi_2$, which is the blow-up of the base-points of $(\tau_1)^{-1}\widetilde\phi^{k-1}$; this yields the following commutative diagram:
$$
\xymatrix{
&&\mathrm{X}_2\ar[lldd]^{\tau_2}\ar[rd]_{\chi_2}&\\
&&& \mathrm{X}_1\ar[rd]^{\chi_1}\ar[ld]_{\tau_1}&&\\
\widetilde{ \mathrm{S}}\ar@{-->}[rr]_{\widetilde\phi^{k-1}}&&\widetilde{ \mathrm{S}}\ar@{-->}[rr]_{\widetilde\phi}&& \widetilde{\mathrm{S}}.
}$$
Because $p$ is not a base-point of $\widetilde{\phi}^{-k}$, the curve $C_1\subset \mathrm{X}_1$ has to be contracted by $\tau_2\circ (\chi_2)^{-1}$, and the curve $C_2=(\chi_2)^{-1}(C_1)$ has self-intersection $-1$ (otherwise $\widetilde{\phi}^{-k}$ would have a base-point infinitely near to $p$, and $p$ itself would then be a base-point). 

The curve $\tau_1(C_1)$ is contracted by $\widetilde{\phi}$, and is then not a $(-1)$-curve. Because it is contracted by $\widetilde\phi^{1-k}$, there is a base-point $q$ of $\widetilde\phi^{1-k}$ that is a proper point of $\tau_1(C_1)$. Since $\tau_2$ is a morphism, $\tau_1\chi_2$ blows-up all base-points of $\widetilde{\phi}^{k-1}$, and thus blows-up $q$.
The fact that $q$ is a base-point of $\widetilde\phi^{1-k}$, implies that it is not a base-point of $\widetilde\phi$, and thus that it is not blown-up by $\tau_1$. In consequence, $(\tau_1)^{-1}(q)$ is a point blown-up by $\xi_2$, and which lies on $C_1$. This is incompatible with the fact that $C_2=(\chi_2)^{-1}(C_1)$ has  self-intersection $-1$.
 \end{proof}
 
\begin{rem}
In \cite[Theorem 0.4]{DiFa} one can find a characterisation of hyperbolic birational maps $\phi$ which are conjugate to an automorphism of a projective surface. If $\phi\in \Bir(\p^2)$ is hyperbolic, we conjugate it to a birational map of a smooth projective surface $\mathrm{S}$ where the action is algebraically stable (this means that the $(\phi|_{\mathrm{H}^{1,1}(\mathrm{S})})^n=(\phi^n)|_{\mathrm{H}^{1,1}(\mathrm{S})}$ for each $n$); its action on $\mathrm{H}^{1,1}(\mathrm{S})$ admits the eigenvalue $\lambda(\phi)>1$ with eigenvector $\theta_+$. The map $\phi$ is birationally conjugate to an automorphism if and only if $(\theta_+)^2=0$. 

Proposition~\ref{Pro:RecCorMuInvariant} gives another characterisation, for all maps $\phi\in \Bir(\p^2)$ (not only hyperbolic maps), depending only on $\mu(\phi)$. 
\end{rem}

\begin{eg}
In \cite{BK1, BK2, BK3, BK4, DeGr},  automorphisms with positive entropy are constructed starting from a birational map of $\p^2$. In \cite{DeGr} the authors take a birational map $\psi\in \Bir(\p^2)$, and choose $A\in \Aut(\p^2)$ such that $A\psi$ is conjugate to an automorphism of a surface with dynamical degree $>1$. The way to find $A$ is exactly to ensure that $A\psi$ has no persistent base-point (\emph{i.e.} $\mu(A\psi)=0$).  Let us give an example (\cite{DeGr}):

Let $\phi=A\psi$ be the birational map given by 
$$A\colon(x:y:z)\dashrightarrow\big(\alpha x+2(1-\alpha)y+(2+\alpha-\alpha^2)z:-x+(\alpha+1)z:x-2y+(1-\alpha)z\big)$$ with 
$\alpha\in\mathbb{C}\setminus\{0,\, 1\}$ and $$\psi\colon(x:y:z)\dashrightarrow(xz^2+y^3:yz^2:z^3).$$  The map $\psi$ (resp. 
$\psi^{-1}$) has five base-points, $p=(1:0:0)$ and four points infinitely near; we will denote $\widehat{P}_1$ (resp. $\widehat{P}_2$) 
the collection of these points. The automorphism $A$ is chosen such that:
\begin{itemize}
\item[$\bullet$] $\widehat{P}_1$, $A\widehat{P}_2$, $A\psi A\widehat{P}_2$ have distinct supports;
\item[$\bullet$]  $\widehat{P}_1=(A\psi)^2A\widehat{P}_2$.
\end{itemize}
In particular the base-points of~$\phi$ are non-persistent, so $\phi$ is conjugate to an automorphism of a
rational surface. More precisely $\phi$ is conjugate to an automorphism with positive entropy on $\mathbb{P}^2$ blown up in 
$\widehat{P}_1$, $A\widehat{P}_2$ and $A\psi A\widehat{P}_2$ (\emph{see} \cite[Theorem 3.1]{DeGr}).
\end{eg}

\section{Growth of Jonqui\`eres twists}\label{Sec:Jonq}

\begin{lem}\label{Lem:Jon}
Let $\phi$ be a birational map of $\mathbb{P}^2$ which preserves the pencil of lines passing through some point~$p_0$.  The set 
$$\left\{\deg\phi^k- k\cdot \frac{\mu(\phi)}{2}\ \Big\vert\ k\ge 0\right\}\subset \mathbb{Z}$$
 is bounded.
 
 In particular, the sequence $\left\{\deg\phi^k\right\}_{k\in \mathbb{N}}$ grows linearly if and only if $\mu(\phi)>0$ and its growth 
is given by  $$\frac{\mu(\phi)}{2}\in \frac{1}{2}\mathbb{N}.$$
\end{lem}

\begin{rem}
Conjugating $\phi$ by a map which preserves the pencil does not change the growth $\left\{\deg\phi^k\right\}_{k\in \mathbb{N}}$, but 
conjugating it by a map which does not preserve the pencil can increase it (\emph{see} Proposition~\ref{pro:MinimalJonq}).
\end{rem}

 \begin{proof}
 For any $k$,  $\phi^k$ preserves the pencil of lines passing through $p_0$. It implies that the linear system of~$\phi^k$ (which is 
the pull-back of the system of lines of $\p^2$ by $\phi^k$) has multiplicity $\deg\phi^k-~1$ at $p_0$ and has exactly $2(\deg\phi^k-1)$ 
other base-points, all of multiplicity $1$. In particular, $\deg\phi^k=\lfloor \frac{\b(\phi^k)}{2}\rfloor$. The result follows then 
directly from Proposition~\ref{Pro:MuPersists}.
 \end{proof}

\begin{eg}
Let us consider the family of birational maps studied in \cite{De2} and defined as 
\begin{align*}
& f_{\alpha,\beta}\colon(x:y:z) \dashrightarrow((\alpha x+y)z:\beta y(x+z):z(x+z)), && \alpha, \, \beta \in \mathbb{C}^*.
\end{align*}
Any of the $f_{\alpha,\beta}$ has three base-points: $(1:0:0)$, $(0:1:0)$ and $(-1:\alpha:1)$, and preserves the pencil of lines 
passing through $(1:0:0)$. Checking that $(-1:\alpha:1)$ is the only one persistent base-point (\cite[Theorem~1.6]{De2}), 
the growth of $\big\{\deg f_{\alpha,\beta}^k\big\}_{k\in\mathbb{N}}$ is given by $\frac{k}{2}$ (\emph{see} \cite[Lemma~1.4]{De2}).
\end{eg} 

\begin{pro}\label{pro:MinimalJonq}
Let $\phi\in \Bir(\p^2)$ be a Jonqui\`eres twist. There exists an integer $a\in \mathbb{N}$ such that 
$$\lim\limits_{k\to +\infty} \frac{\deg(\phi^k)}{k}=a^2\frac{\mu(\phi)}{2}.$$
Moreover, $a=1$ if and only if $\phi$ preserves a pencil of lines.
\end{pro}

\begin{proof}
Since $\phi$ is a Jonqui\`eres twist, there exists $\psi\in \Bir(\p^2)$ such that $\widetilde{\phi}=\psi\phi\psi^{-1}$ preserves the 
pencil of lines passing through some point $p\in \p^2$. Let $\pi\colon \mathbb{F}_1\to \p^2$ be the blow-up of $p\in \mathbb{P}^2$, 
and let $\widehat{\phi}\in \Bir(\mathbb{F}_1)$ be~$\widehat{\phi}=\pi^{-1}\widetilde{\phi}\pi$. Denote by $L_{\p^2}$ the linear system of lines 
of $\p^2$ and by $\Lambda$ the linear system on $\mathbb{F}_1$ corresponding to the image by $\pi^{-1}\psi$ of the system of lines of 
$\p^2$. The degree of $\phi^k$ is equal to the free intersection of $L_{\p^2}$ with $\phi^k(L_{\p^2})$, which is the free intersection 
of $\Lambda$ with $\widehat{\phi}(\Lambda)$.

On $\mathbb{F}_1$, $\Lambda$ is linearly equivalent to $aL+bf$, where $L=\pi^{-1}(L_{\p^2})$, $f$ is the divisor of a fibre and where 
$a$, $b\in~\mathbb{N}$. Note that $\phi^k(L)$ is the transform on $\mathbb{F}_1$ of the linear system of $\widetilde{\phi}^{-k}$, and is 
thus equal to $L+(d_k-1)f$, where $d_k$ is the degree of $\widetilde{\phi}^k$ (and of $\widetilde{\phi}^{-k}$).
The system $\widehat{\phi}^k(\Lambda)$ is then linearly equivalent to $aL+(a(d_k-1)+b)f$, so the total intersection of $\widehat{\phi}^k
(\Lambda)$ with $\Lambda$ is $a^2 d_k+2ab.$ Because $\widehat{\phi}^k(\Lambda)\cdot f=a$, each base-point of $\widehat{\phi}^k(\Lambda)$ has 
at most multiplicity $a$. By Lemma~\ref{Lem:Jon}, $\displaystyle\lim\limits_{k\to +\infty} \frac{d_k}{k}=\frac{\mu(\phi)}{2}$.
 The number of base-points of $\Lambda$ being bounded, the free intersection of $\widehat{\phi}^k(\Lambda)$ with $\Lambda$ grows like 
$a^2\frac{\mu(\phi)}{2}\cdot k$. 
 
 It remains to see that if $a=1$ then $\phi$ preserves a pencil (the other direction follows from Lemma~\ref{Lem:Jon}). If $a=1$, one gets
 $f\cdot \Lambda=1$. This implies that the free intersection of $\phi^{-1}\pi(f)$ with $\phi^{-1}\pi(\Lambda)=L_{\p^2}$ is $1$; so 
$\phi^{-1}\pi(f)$ is a pencil of lines, invariant by $\phi$.
\end{proof}

Lemma \ref{Lem:Jon} and the second assertion of Corollary \ref{Cor:MuInvariant} imply the following statement of \cite[Theorem 0.2.]{DiFa}
 
\begin{cor}\label{Cor:difa}
Let  $\phi$ be a Jonqui\`eres twist; then $\phi$ is not conjugate to an automorphism.
\end{cor}  
 
 We can also derive the following new results.
\begin{cor}\label{Cor:conj}
Let $\phi$ be a Jonqui\`eres twist. If $\phi^m$ and $\phi^n$ are conjugate in $\Bir(\p^2)$ for some $m$, $n\in \mathbb{Z}$, then 
$\vert m\vert=\vert n\vert$.
\end{cor}

\begin{proof}
The fact that $\left\{\deg \phi ^k\right\}_{k\in \mathbb{N}}$ grows linearly implies that $\phi$ preserves a pencil of rational 
curves \cite[Theorem 0.2]{DiFa}. In particular $\phi$ is conjugate to a birational map of $\p^2$ which preserves the pencil of 
lines passing through some fixed point $p_0$. 
According to Lemma~\ref{Lem:Jon}, one finds $\mu(\phi)>0$.

As $\phi^m$ and $\phi^n$ are conjugate in $\Bir(\p^2)$ one has $\mu(\phi^m)=\mu(\phi^n)$ (Corollary~\ref{Cor:MuInvariant}). 
Sin\-ce~$\mu(\phi^k)=~\vert k\cdot \mu(\phi)\vert$ for any $k$, we get $\vert m\vert=\vert n\vert$.
\end{proof}

\section{Growth of Halphen twists}\label{Sec:Halphen}

In Section~\ref{Sec:Jonq} (especially Lemma~\ref{Lem:Jon}), we described the degree growth of a Jonqui\`eres twist $\phi$, and showed 
that it is given by $\mu(\phi)$, a birational invariant given by the growth of base-points. For an Halphen twist, the dynamical number 
of base-points is trivial, but the growth can also be quantified by an invariant. Recall that an Halphen twist $\phi$ preserves an unique pencil of elliptic curves. By \cite[Theorem 2 and Proposition 7, page 127]{Gi}, a power of $\phi$ preserves any member of the pencil, and acts on this via a translation. 
\begin{pro}\label{Pro:Halphen}
Let $\phi\in \Bir(\p^2)$ be an Halphen twist.

\begin{enumerate}
\item
The set $$\left\{\lim\limits_{k\to +\infty} \frac{\deg(\psi\phi^k\psi^{-1})}{k^2}\ \big\vert\ \psi\in \Bir(\p^2)\right\}$$ admits a minimum, 
which is a positive rational number $\kappa(\phi)\in \mathbb{Q}$. If $\phi$ acts via a translation on each member of its invariant pencil of elliptic curves, then $\kappa(\phi)\in 9\mathbb{N}$.
\item

There exists an integer $a\ge 3$ such that $\lim\limits_{k\to +\infty} \frac{\deg(\phi^k)}{k^2}=\kappa(\phi)\cdot\frac{a^2}{9}$.
\item

 The following conditions are equivalent:

\begin{enumerate}
\item
$a=3$;
\item
$\phi$ preserves an Halphen pencil, \emph{i.e.}\ a pencil of $($elliptic$)$ curves of degree $3n$ passing through $9$ points with 
multiplicity $n$.
\end{enumerate}\end{enumerate}
\end{pro}

\begin{proof}
An Halphen twist preserves an unique elliptic fibration, so there exists an element $\psi$ in $\Bir(\p^2)$ such that $\phi'=\psi \phi 
\psi^{-1}$ preserves an Halphen pencil. Denoting by $\pi\colon \mathrm{S}\to \p^2$ the blow-up of the $9$ base-points of the pencil, 
$\widehat{\phi}=\pi^{-1}\phi'\pi$ is an automorphism of $\mathrm{S}$, which preserves the elliptic fibration $\mathrm{S}\to \p^1$ given by 
$\vert-mK_\mathrm{S}\vert$ for some positive integer $m$. 

 
  Replacing $\widehat{\phi}$ by some power if needed, we can assume that
 $\widehat{\phi}$ is a translation on a general fibre.  As explained in the proof of \cite[Proposition 9, page 132]{Gi}, this yields the existence of
an element~$\Delta\in \mathrm{Pic}(\mathrm{S})$ (depending on $\widehat{\phi}$) with $\Delta\cdot K_\mathrm{S}=0$ such that the action of $\widehat{\phi}$ on 
$\mathrm{Pic}(\mathrm{S})$ is given by
$$D\mapsto D-m(D\cdot K_\mathrm{S})\cdot \Delta+\gamma K_\mathrm{S},$$
where $\gamma$ is an integer depending on $D$ which can be computed using the self-intersection: $$\gamma=-\frac{m^2}{2}(D\cdot
 K_\mathrm{S})\cdot \Delta^2+m(D\cdot \Delta).$$ 

We denote by $L$ the linear system of lines of $\p^2$ and by $\Lambda=\pi^{-1}\psi(L)$ its  transform on $\mathrm{S}$. The degree 
of $\phi^n$ is equal to the free intersection of $L$ with $\phi^n(L)$, which is equal to the free intersection of $\Lambda$ with 
$\widehat{\phi}^n(\Lambda)$.

The map $\widehat{\phi}^n$ acts on $\mathrm{Pic}(\mathrm{S})$ as 
$$D\mapsto D-m(D\cdot K_\mathrm{S})\cdot (n\Delta)+\left(-\frac{m^2}{2}(D\cdot K_\mathrm{S})\cdot (n\Delta)^2+m(D\cdot (n\Delta))\right) 
K_\mathrm{S}.$$
This yields 
$$\begin{array}{rcl}
\Lambda\cdot \widehat{\phi}^n(\Lambda)&=&
\Lambda^2-m(\Lambda\cdot K_\mathrm{S})\cdot (n\Delta\cdot \Lambda)+\left(-\frac{m^2}{2} (\Lambda\cdot K_\mathrm{S})(n\Delta)^2+m(n 
\Lambda\cdot \Delta)\right)(K_\mathrm{S}\cdot \Lambda)\\
&=&
\Lambda^2+\left(-\frac{m^2}{2} (\Lambda\cdot K_\mathrm{S})^2(\Delta)^2\right)\cdot n^2.
\end{array}$$
The free intersection between $\Lambda$ and $\widehat{\phi}^n(\Lambda)$ is thus equal to $$\Lambda^2+\left(-\frac{m^2}{2} (\Lambda\cdot 
K_\mathrm{S})^2(\Delta)^2\right)\cdot n^2-\sum_{i=1}^k \mu_i(\Lambda)\cdot \mu_i(\widehat{\phi}^n(\Lambda)),$$ where $\mu_i(\Lambda)$ and 
$\mu_i(\widehat{\phi}^n(\Lambda))$ denote the multiplicities of respectively $\Lambda$ and $\widehat{\phi}^n(\Lambda)$ at the $r$ base-points 
of $\Lambda$. Since $\widehat{\phi}$ is an automorphism of $\mathrm{S}$, the contribution given by the base-points is bounded, so we find 
that $$\lim\limits_{n\to +\infty} \frac{\deg(\phi^n)}{n^2}={m^2}(\Lambda\cdot K_\mathrm{S})^2\cdot\left(\frac{-\Delta^2}{2}\right).$$
Note that $\Lambda$ is the lift by $\pi^{-1}$ of the homaloidal linear system $\psi(L)$, so $\Lambda\cdot (-K_\mathrm{S})\ge 3$, and 
equality holds if and only if $\Lambda$ has no base-point. This shows that the minimum among all homaloidal systems is attained when 
$\Lambda$ has no base-point; we get $\kappa(\phi)=9{m^2}\cdot\left(\frac{-\Delta^2}{2}\right)$ and $a=\Lambda\cdot K_\mathrm{S}$.

Let us prove that $-\Delta^2$ is a positive even number. To do this, we take the orthogonal basis $L,E_1,\dots,E_9$ of $\Pic{\mathrm{S}}$, where $L$ is the pull-back of a line of $\p^2$ and $E_1,\dots,E_9$ are the exceptional divisors associated to the points blown-up. Writing $$\Delta=dL-\sum_{i=1}^9 a_i E_i,$$ the equality $\Delta \cdot K_\mathrm{S}=0$ implies that $\sum a_i=3d$. Modulo $2$, we have $\sum (a_i)^2\equiv \sum a_i\equiv 3d\equiv d^2$, so $\Delta^2=d^2-\sum (a_i)^2$ is even. Applying Cauchy-Schwarz to $(1,\dots,1)$ and $(a_1,\dots,a_9)$, we get $9\sum (a_i)^2\ge \left(\sum a_i\right)^2$, and equality holds only when all $a_i$ are equal. This latter would imply that $\Delta$ is a multiple of $K_\mathrm{S}$, and thus that  $\widehat{\phi}$ acts trivially on $\mathrm{Pic}(\mathrm{S})$. Hence we have $9\sum (a_i)^2> \left(\sum a_i\right)^2=(3d)^2$, which implies that $-\Delta^2=\sum (a_i)^2-d^2>0$.

The fact that $-\Delta^2$ is a positive even number implies that $\kappa(\phi)=9{m^2}\cdot\left(\frac{-\Delta^2}{2}\right)$ is a positive integer divisible by $9$. The equality 
$\kappa(\phi)=\lim\limits_{k\to +\infty} \frac{\deg(\phi^k)}{k^2}$ is equivalent to the fact that $\Lambda$ has no base-point, which 
corresponds to say that~$\psi^{-1}\pi$ is a birational morphism, or equivalently that the unique pencil of elliptic curves invariant 
by $\phi$ is an Halphen pencil. 
\end{proof}

\begin{cor}\label{Cor:ConjHalphen}
Let $\phi\in \Bir(\p^2)$ be an Halphen twist. The integer $\kappa(\phi)$ is a birational invariant which satisfies $\kappa(\phi^m)=
m^2\kappa(\phi)$ for any $m\in\z$. In particular, the maps $\phi^n$ and $\phi^m$ are not conjugate if $\vert m\vert\not=\vert n\vert$.  
\end{cor}

\begin{proof}
Is a direct consequence of Proposition~\ref{Pro:Halphen}.
\end{proof}

Let us give an example where $\kappa(\phi)$ is not an integer.

\begin{eg}
Let $\Lambda$ be the pencil of cubic curves of $\p^2$ given by  $\lambda(x^2y+z^3+y^2z)+\mu(x^2z+y^3+yz^2)=0$, $(\lambda:\mu)\in \p^1$. The pencil is invariant by the automorphism $\alpha\in \Aut(\p^2)$ defined by $(x:y:z)\mapsto (\im x:-y:z)$.
 
 Denote by $\pi\colon \mathrm{S}\to \p^2$ the blow-up of the base-points of $\Lambda$, which are $7$ proper points of $\p^2$ and $2$ infinitely near points. More precisely, the $7$ proper points are
  $$\begin{array}{c}
  p_1=(1:0:0),\\
   p_2=(\im \sqrt{2}:1:1),\, p_3=\alpha(p_2)=(-\sqrt{2}:-1:1),\\
    p_4=\alpha(p_3)=(-\im \sqrt{2}:1:1),\, p_5=\alpha(p_4)=(\sqrt{2}:-1:1),\\
   p_6=(0:\im: 1),\, p_7=\alpha(p_6)=(0:-\im:1).\end{array}$$
The last two points are the following: the point $p_8$ is infinitely near to $p_6$, corresponding to the tangent direction of the line $y=\im z$, and $p_9$ is infinitely near to $p_7$, corresponding to the tangent direction of the line $y=-\im z$.

The surface $\mathrm{S}$ inherits an elliptic fibration $\mathrm{S}\to \p^1$, and the lift of $\alpha$ yields an automorphism $\widehat{\alpha}=\pi\alpha\pi^{-1}$ of~$\mathrm{S}$. Denote by $E_i\in \Pic{\mathrm{S}}$ the divisor of self-intersection $-1$ corresponding to the point $p_i$. If $i\not=6$, $7$, then~$E_i$ corresponds to a $(-1)$-curve of $\mathrm{S}$; and $E_6$, $E_7$ correspond to two reducible curves of $\mathrm{S}$. 

For any $\Delta\in \Pic{\mathrm{S}}$ satisfying $\Delta\cdot K_\mathrm{S}=0$, we denote by $\iota_\Delta\in \Aut(\mathrm{S})$ the automorphism which restricts on a general fibre $C$ to the translation given by the divisor $\Delta\vert_C$.  If $\Delta^2=-2$, the action of $\iota_\Delta$ on $\Pic{\mathrm{S}}$ is given by (\emph{see} the proof of Proposition~\ref{Pro:Halphen})
$$D\mapsto D-(D\cdot K_\mathrm{S})\cdot \Delta+\left(D\cdot
 (K_\mathrm{S}+\Delta)\right) \cdot K_\mathrm{S}.$$
 
 For any automorphism $\sigma\in \Aut(\mathrm{S})$, one can check that $\iota_{\sigma(\Delta)}=\sigma\, \iota_\Delta \sigma^{-1}$. In particular, we have
$$(\widehat{\alpha}\,\iota_\Delta)^4=(\widehat{\alpha}\,\iota_\Delta\,\widehat{\alpha}^{-1})(\widehat{\alpha}^2\,\iota_\Delta\,\widehat{\alpha}^{-2})(\widehat{\alpha}^3\,\iota_\Delta\,\widehat{\alpha}^{-3})\iota_\Delta=\iota_{\widehat{\alpha}(\Delta)}\,\iota_{\widehat{\alpha}^2(\Delta)}\iota_{\widehat{\alpha}^3(\Delta)}\iota_{\Delta}=\iota_{\widehat{\alpha}(\Delta)+\widehat{\alpha}^2(\Delta)+\widehat{\alpha}^3(\Delta)+\Delta}.$$

Because of the action of $\alpha$ on the points $p_i$, we have $$\widehat{\alpha}(E_6)=E_7, \,\widehat{\alpha}(E_7)=E_6, \,\widehat{\alpha}(E_2)=E_3, \,\widehat{\alpha}(E_3)=E_4,\,\widehat{\alpha}(E_4)=E_5, \,\widehat{\alpha}(E_5)=E_2.$$

We now fix $\Delta\in \Pic{\mathrm{S}}$ to be the divisor $E_2-E_6$ (that satisfies $\Delta\cdot K_\mathrm{S}=0$ and $\Delta^2=-2$), and obtain
$$\widehat{\alpha}(\Delta)+\widehat{\alpha}^2(\Delta)+\widehat{\alpha}^3(\Delta)+\Delta=E_2+E_3+E_4+E_5-2E_6-2E_7,$$
which has square $-8$. In particular, the number $\kappa$ associated to 
$(\widehat{\alpha}\iota_\Delta)^4=\iota_{\widehat{\alpha}(\Delta)+\widehat{\alpha}^2(\Delta)+\widehat{\alpha}^3(\Delta)+\Delta}$ is 
$$-9\frac{(\iota_{\widehat{\alpha}(\Delta)+\widehat{\alpha}^2(\Delta)+\widehat{\alpha}^3(\Delta)+\Delta})^2}{2}=36.$$
This shows, by Corollary~\ref{Cor:ConjHalphen}, that $\kappa(\phi)=\frac{9}{4}$, where $\phi$ is the birational map of $\p^2$ conjugate to $\widehat{\alpha}\,\iota_\Delta$ by $\pi^{-1}$, namely $\alpha\pi\,\iota_\Delta\pi^{-1}$.
\end{eg}

\section{Applications}\label{Sec:applications}

\subsection{Birational maps having two conjugate iterates}

\begin{lem}\label{Lem:EmbBSGrowth}
Let $\phi$ denote a birational map of $\mathbb{P}^2$. Assume that $\phi^n$ and $\phi^m$
 are conjugate and assume that $\vert m\vert\not=\vert n\vert$.
Then, $\phi$ is an elliptic and satisfies $\lambda(\phi)=1$ and $\mu(\phi)=0$. 
\end{lem}

\begin{proof}
The map $\phi^m$ is conjugate to $\phi^n$ in $\Bir(\p^2)$ so one gets $\lambda(\phi)^{\vert m\vert}=\lambda(\phi)^{\vert n\vert}$ and 
${\vert m\vert}\cdot \mu(\phi)={\vert n\vert}\cdot \mu(\phi)$. This yields $\lambda(\phi)=1$ and $\mu(\phi)=0$.

The fact that $\lambda(\phi)=1$ implies that $\phi$ is elliptic, or a Jonqui\`eres or Halphen twist. The Jonqui\`eres and Halphen cases are impossible (Corollaries~\ref{Cor:conj} and~\ref{Cor:ConjHalphen}).
\end{proof}

\begin{pro}\label{Pro:ConjYplus1}
Let $\phi$ denote a birational map of $\mathbb{P}^2$ of infinite order. Assume that $\phi^n$ and $\phi^m$
 are conjugate and assume that $\vert m\vert\not=\vert n\vert$. Then, $\phi$ is conjugate to an automorphism of $\C^2$ of the form  
$(x,y)\mapsto (\alpha x,y+1)$, where $\alpha\in \C^{*}$ such that $\alpha^{m+n}=1$ or $\alpha^{m-n}=1$.
 
In particular, if $\phi$ is conjugate to $\phi^n$ for any positive integer $n$, then $\phi$ is conjugate to $(x,y)\mapsto (x,y+1)$.
 \end{pro}
 
\begin{proof}
Follows from Corollary~\ref{Cor:Lin}, Corollary~\ref{Cor:ConjHalphen} and Lemma~\ref{Lem:EmbBSGrowth}.
\end{proof}

\subsection{Morphisms of Baumslag-Solitar groups in the Cremona group}\label{Subsec:Baumslag}

For any integers $m$, $n$ such that $mn\not=0$, the Baumslag-Solitar group~$\mathrm{BS}(m,n)$ is defined by the following presentation
\begin{align*}
&\mathrm{BS}(m,n)=\langle r,\, s\,\vert\, rs^mr^{-1}=s^n\rangle.
\end{align*}

Recall that if $\mathrm{G}$ is a group, the derived groups of $\mathrm{G}$ are
\begin{align*}
&\mathrm{G}^{(0)}=\mathrm{G}, && \mathrm{G}^{(i)}=\left[\mathrm{G}^{(i-1)},\mathrm{G}^{(i-1)}\right]=\langle ghg^{-1}h^{-1}\,\vert\,g,
\,h\in\mathrm{G}^{(i-1)}\rangle \,\text{ for all $i\geq 1$},
\end{align*}
and  that $\mathrm{G}$ is \emph{solvable} if there exists an integer $N$ such that $\mathrm{G}^{(N)}=\{\mathrm{id}\}$. 

\medskip

The groups $\mathrm{BS}(m,n)$ (resp. the subgroups of finite index of $\mathrm{BS}(m,n)$) are solvable if and only if $\vert m\vert=1$ 
or $\vert n\vert =1$ (\emph{see} \cite[Proposition A.6]{So}).

\medskip

%

A group $\mathrm{G}$ is said to be {\it residually finite} if for any $g$ in $\mathrm{G}\setminus\{\mathrm{id}\}$ there exist a finite 
group $\mathrm{H}$ and a group homomorphism $\Theta \colon\mathrm{G}\to\mathrm{H}$ such that $\Theta(g)$ belongs to 
$\mathrm{H}\setminus\{\mathrm{id}\}$. The group $\mathrm{BS}(m,n)$ is residually finite if and only if $\vert m\vert=1$ or $\vert 
n\vert=~1$ or $\vert m\vert=\vert n\vert$ (\emph{see} \cite{Mes}). Let $V$ be an affine algebraic variety; according to~\cite{BaLu} any 
subgroup of finite index of the automorphisms group of $V$ is residually finite. Therefore if $\vert m \vert\not=\vert n\vert$ and 
$\vert m \vert,\vert n\vert\not=1$  there is no embedding of $\mathrm{BS}(m,n)$ into the group of polynomial automorphisms of the plane.
 There is an other proof using the amalgated structure of the group of polynomial automorphisms of the plane and the fact that 
$\mathrm{BS}(m,n)$ is not solvable (\cite[Proposition 2.2]{CaLa}).

\begin{lem}\label{Lem:NoBoundBS}
Let $\rho$ be a homomorphism from $\mathrm{BS}(m,n)=\langle r,\, s\,\vert\, rs^mr^{-1}=s^n\rangle$ to $\mathrm{Bir}(\mathbb{P}^2)$. 
Assume that~$\vert m\vert$,$\vert n\vert$ and $1$ are distinct. If $\rho(s)$ has infinite order, the image of the subgroup of finite 
index $\langle r,\,s^m\,\vert\,rs^{m^2}r^{-1}=s^{mn}\rangle$ of $\mathrm{BS}(m,n)$ is solvable.
\end{lem}

\begin{proof}By Proposition~\ref{Pro:ConjYplus1}, one can conjugate $\rho$ so that  $\rho(s)\colon(x,y)\dashrightarrow(\alpha x,y+1)$ 
where $\alpha\in \C^*$ and $\alpha^{m-n}=~1$ or $\alpha^{m+n}=~1$. Denoting respectively by $\psi$ the map $\left(x,\frac{n}{m}y\right)$ 
or $\left(x^{-1},\frac{n}{m}y\right)$ one has $\psi\rho(s)^m\psi^{-1}=\rho(s)^n$. So~$\rho(r)=\psi \tau $ where~$\tau$ commutes with 
$\rho(s)^m\colon(x,y)\mapsto(\alpha^m x,y+m)$. 

According to Lemma~\ref{Lem:centtransl}, one then has $\tau=\left(\eta(x),y+R(x)\right)$ for some $\eta\in \PGL(2,\C), 
\eta(\alpha^m x)=\alpha^m \eta(x)$, and some $R\in \C(x)$ satisfying $R(\alpha^m x)=R(x)$. And one gets $\rho(r)=\left(\eta(x)^{\pm 1},
\frac{n}{m}(y+R(x))\right)$.  

 The group generated by $\rho(s^m)$ and $\rho(r)$ is thus solvable.
\end{proof}

\begin{cor}\label{cor:BS}
If $\vert m\vert $, $\vert n\vert$ and $1$ are distinct, there is no embedding of $\mathrm{BS}(m,n)$ into the Cremona group.
\end{cor}

\begin{proof}
Lemma~\ref{Lem:NoBoundBS} shows that the image of any embedding would be virtually solvable, impossible when  $\vert m\vert $, $\vert n\vert$ and $1$ 
are distinct.
\end{proof}

\subsection{Embeddings of $\mathrm{GL}(2,\mathbb{Q})$ into the Cremona group}\label{Subsec:GL2Q}To simplify the notation, we will denote 
in this last section by $(\phi_1(x,y),\phi_2(x,y))$ the rational map $(x,y)\dasharrow (\phi_1(x,y),\phi_2(x,y))$ from $\C^2$ to $\C^2$.

Let us first give examples of embeddings of $\mathrm{GL}(2,\mathbb{Q})$ into the Cremona group.

\begin{eg}\label{Ex:GL2Q}
Let $k$ be an odd integer and let $\chi\colon \mathbb{Q}^{*}\to \mathbb{C}^{*}$ be a homomorphism such that $a\mapsto\frac{\chi(a^2)}{a^k}$ is injective. The morphism $\rho$ 
from $\mathrm{GL}(2,\mathbb{Q})$ to the Cremona group given by
$$\rho\left(\left[\begin{array}{cc}a & b\\ c & d\end{array}\right]\right)=\left(x\cdot \frac{\chi(ad-bc)}{(cy+d)^k},\frac{ay+b}{cy+d}
\right)$$
is an embedding. Note that $\rho(\mathrm{GL}(2,\mathbb{Q}))$ is conjugate to a subgroup of automorphisms of the $k$-th Hirzebruch 
surface $\mathbb{F}_k$. Changing $k$ gives then infinitely many non conjugate embeddings in the Cremona group.
\end{eg}
\begin{rem}Taking $k=1$ and $\chi$ the trivial map, Example~\ref{Ex:GL2Q} yields the embedding $\rho\colon \GL(2,\mathbb{Q})\to \Bir(\p^2)$ given by 
$$\rho\left(\left[\begin{array}{cc}a & b\\ c & d\end{array}\right]\right)=\left( \frac{x}{cy+d},\frac{ay+b}{cy+d}
\right),$$
which is obviously conjugate (by extending the actions to $\p^2$) to the classical embedding 
$$\widehat{\rho}\left(\left[\begin{array}{cc}a & b\\ c & d\end{array}\right]\right)=\left( {ax+by},{cx+dy}
\right).$$

\end{rem}

\begin{thm}\label{thm:GL2Q}
Let $\rho\colon \GL(2,\mathbb{Q})\to \Bir(\p^2)$ be an embedding of $\mathrm{GL}(2,\mathbb{Q})$ into the Cremona group, then up to 
conjugation $\rho$ is one of the embeddings described in Example~$\ref{Ex:GL2Q}$.
\end{thm}

\begin{proof}
Let us set 
\begin{align*}
& \mathrm{t}_q=\left[\begin{array}{cc}1 & q\\ 0 & 1\end{array}\right] &&  \&&&\mathrm{d}_{m,n}=\left[\begin{array}{cc}m & 0\\ 0 & n
\end{array}\right], && q,\, m,\, n\in \mathbb{Q}.
\end{align*}
Remark that $\mathrm{t}_1$ is conjugate to $\mathrm{t}_n$ in~$\mathrm{GL}(2,\mathbb{Q})$, for any $n\in\mathbb{Z}\smallsetminus\{0\}$; 
we can then assume, after conjugation, that $\rho(\mathrm{t}_1)=(x,y+1)$ (Proposition \ref{Pro:ConjYplus1}). As~$\rho(\mathrm{t}_{1/n})$
 commutes with~$\rho(\mathrm{t}_1)$ there exist $A_n$ in $\mathrm{PGL}(2,\mathbb{C})$ and $R_n$ in~$\mathbb{C}(x)$ such that  
$$\rho(\mathrm{t}_{1/n})=(A_n(x),y+R_n(x))$$
(\emph{see} Lemma \ref{Lem:centtransl}). Let us prove now that $A_n(x)=x$.
Since $\mathrm{t}_{1/n}^n=\mathrm{t}_1$ the element $A_n$ is of finite order so~$A_n$ is conjugate to some $\xi x$ where $\xi$ is some 
root of unity. Hence $\rho(\mathrm{t}_{1/n})$ is conjugate to $(\xi x, y+Q(x))$ where~$Q\in~\C(x)$ satisfies $Q(x)+Q(\xi x)+\dots 
+Q(\xi^{n-1}x)=1$. The map $(\xi x, y+Q(x))$  is then conjugate to $\left(\xi x,y+\frac{1}{n}\right)$ by $\left(x,y-
\frac{\sum_{i=1}^{n-1} iQ(\xi^i x)}{n}\right)$. Since $\mathrm{t}_{1/n}$ is conjugate to $\mathrm{t}_1$,
 Proposition~\ref{Prop:ConjDiagAlmostDiag} implies that 
$\xi=1$, which achieves to show that $A_n(x)=x$. This implies, with equality $\rho(\mathrm{t}_{1/n})^n=\rho(\mathrm{t}_1)$, that 
$R_n(x)=1/n$. We thus have for any $q$ in $\mathbb{Q}$ $$\rho(\mathrm{t}_q)=(x,y+q).$$

From $\mathrm{d}_{m,n}\mathrm{t}_1\mathrm{d}_{m,n}^{-1}=\mathrm{t}_{m/n}$ one gets (using again Lemma \ref{Lem:centtransl}) that 
\begin{align*}
& \rho(\mathrm{d}_{m,n})=\left(\eta_{m,n}(x),\frac{m}{n}y+R_{m,n}(x)\right), && \eta_{m,n}\in\mathrm{PGL}(2,\mathbb{C}),\, 
R_{m,n}\in\mathbb{C}(x).
\end{align*}
The map $(\mathbb{Q}^{*})^2\to \PGL(2,\C)$ given by $(m,n)\mapsto \eta_{m,n}$ is a homomorphism, which cannot be injective. There exists 
thus one element $\mathrm{d}_{m,n}$ with $(m,n)\not=(1,1)$ such that $\rho(\mathrm{d}_{m,n})=\left(x,\frac{m}{n}y+R_{m,n}(x)\right).$
Note that $m\not=n$ since the centralizers of $\rho(\mathrm{d}_{m,n})$ and $\rho(\mathrm{t}_1)$ are different. Conjugating by 
 $\left(x,y+\frac{R_{m,n}(x)}{m/n-1}\right)$, we can assume that $\rho(\mathrm{d}_{m,n})=\left(x,\frac{m}{n}y\right)$. From 
$\mathrm{d}_{m,n}\mathrm{d}_{a,b}=\mathrm{d}_{a,b}\mathrm{d}_{m,n}$ one gets for any $a$, $b$ in $\mathbb{Q}$
\begin{align*}
& \rho(\mathrm{d}_{a,b})=\left(\eta_{a,b}(x),\frac{a}{b}y\right), && \eta_{a,b}\in\mathrm{PGL}(2,\mathbb{C}).
\end{align*}
The homomorphism $\mathbb{Q}^{*}\to \PGL(2,\C)$ given by $a\mapsto \eta_{a,a}$ is injective, so up to conjugation by an element 
of~$\PGL(2,\C)$ we can assume that for any $a\in \mathbb{Q}\smallsetminus\{0,1\}$ there exists $\chi_{a,a}\in \C\smallsetminus\{0,1\}$ such 
that $\eta_{a,a}(x)=\chi_{a,a} x$. This implies the existence of $\chi_{a,b}\in \C^{*}$ for any $a$, $b$ in $(\mathbb{Q}^{*})^2$, 
such that $\eta_{a,b}(x)=\chi_{a,b} x$.

We now compute the image of $M=\left[\begin{array}{cc}0 &1\\-1&0\end{array}\right]$. Since $\rho(M)$ commutes with 
$\rho(\mathrm{d}_{2,2})=(\chi_{2,2}x,y)$ where $\chi_{2,2}\in \C^{*}$ is of infinite order, there exist $R\in \C(y)$ and $\nu\in 
\PGL(2,\C)$  such that $\rho(M)=(xR(y),\nu(y))$ (Lemma~\ref{Lem:centtransl}). For any $a\in \mathbb{Q}^{*}$, equality 
$M \mathrm{d}_{a,1}=\mathrm{d}_{1,a}M$ yields $$\big(\chi_{a,1}x\cdot R(ay),\nu(ay)\big)=\left(\chi_{1,a}x\cdot R(y),\frac{1}{a}\nu(y)
\right).$$ This implies  that  $R(y)=\alpha y^{-k}$ and $\nu(y)=\frac{\beta}{y}$, for some $\alpha$, $\beta\in \C^{*}$, $k\in \mathbb{Z}$,
 \emph{i.e.} $\rho(M)=\left(\alpha \frac{x}{y^k},\frac{\beta}{y}\right)$. We use now equality $(M\mathrm{t}_1)^3=\mathrm{id}$: the 
second component of $(\rho(M)\rho(\mathrm{t}_1))^3$ being $$\frac{\beta(y+\beta+1)}{(\beta+1)y+2\beta+1},$$ we find $\beta=-1$ and 
compute $(\rho(M)\rho(\mathrm{t}_1))^3=(x\alpha^3 (-1)^k,y)$ so $\alpha^3=(-1)^k$. Since $\rho(M)^2=(x\alpha^2(-1)^k,y)$ has order $2$,
 we have $-\alpha^2=(-1)^k$, thus $\alpha=-1$ and $k$ is odd. 

Writing $\chi(a)=\chi_{a,1}$ for any $a\in \mathbb{Q}$, the map $\mathbb{Q}^{*}\to \mathbb{C}^{*}$ given by $a\to \chi(a)$ is a homomorphism, and one gets $\rho(\mathrm{d}_{a,1})=(\chi(a)x,ay)$. The group $\mathrm{GL}(2,\mathbb{Q})$ is generated by 
the maps $\mathrm{d}_{a,1}, \mathrm{t}_a$ and $M$, so
$$\rho\left(\left[\begin{array}{cc}a & b\\ c & d\end{array}\right]\right)=\left(x\cdot \frac{\chi(ad-bc)}{(cy+d)^k},\frac{ay+b}{cy+d}
\right),$$
for any $\left[\begin{array}{cc}a & b\\ c & d\end{array}\right]\in \GL(2,\mathbb{Q})$. This yields an embedding if and only if the homomorphism from $\mathbb{Q}^*$ to $\mathbb{C}^*$ given by $a \mapsto \frac{\chi(a^2)}{a^k}$ is injective.

It remains to observe that $k$ can be chosen to be positive. Indeed, otherwise one conjugates by $\left(\frac{1}{x},y\right)$ and 
replaces $\chi$ with $\frac{1}{\chi}$ to replace $k$ with $-k$.
\end{proof}

One can see that $\GL(n,\mathbb{Q})$ does not embedd into $\Bir(\p^2)$ as soon as $n\ge 3$. Indeed, Theorem~\ref{thm:GL2Q} implies that 
the diagonal matrices are sent onto diagonal elements of $\PGL(3,\C)=\Aut(\p^2)$, which is impossible, by considering the involutions.
 One can also find another less obvious corollary:

\begin{cor}
Let $\rho\colon \mathrm{GL}(2,\mathbb{C})\to \mathrm{Bir}(\mathbb{P}^2)$ be an embedding of $\mathrm{GL}(2,\mathbb{C})$ into the Cremona 
group. There exist a positive odd integer $k$, a field homomorphism $\tau\colon \C\to \C$ and a group homomorphism 
$\chi\colon \C^{*}\to \C^{*}$ such that 
\begin{align*}
&\rho\left(\left[\begin{array}{cc}a & b\\ c & d\end{array}\right]\right)=\left(x\cdot \frac{\chi(ad-bc)}{(\tau(c)y+\tau(d))^k},
\frac{\tau(a)y+\tau(b)}{\tau(c)y+\tau(d)}\right), &&\forall \left[\begin{array}{cc}a & b\\ c & d\end{array}\right]\in \GL(2,\mathbb{C}).
\end{align*}
\end{cor}

\begin{rem}
One sees that in the description above, $\rho$ is an embedding if and only if  the group homomorphism $\C^{*}\to \C^{*}$ given by 
$a\to \frac{\chi(a^2)}{\tau(a)^k}$ is injective. This happens for instance by taking $\chi(a)=\tau(a)^{\frac{k\pm 1}{2}}$, any positive 
odd integer $k$ and any field homomorphism $\tau\colon \C\to \C$.
\end{rem}

\begin{proof}
The map $\rho$ induces an embedding of $\mathrm{GL}(2,\mathbb{Q})$ into the Cremona group. According to Theorem~\ref{thm:GL2Q} one has a
description of $\rho_{\vert \mathrm{GL}(2,\mathbb{Q})}$. Up to conjugacy, there exists an odd positive integer $k$ and an 
homomorphism $\widetilde{\chi}\colon \mathbb{Q}^{*}\to \C^{*}$ such that 
\begin{align*}
&\rho\left(\left[\begin{array}{cc}a & b\\ c & d\end{array}\right]\right)=\left(x\cdot \frac{\widetilde{\chi}(ad-bc)}{(cy+d)^k},
\frac{ay+b}{cy+d}\right), &&\forall \left[\begin{array}{cc}a & b\\ c & d\end{array}\right]\in \GL(2,\mathbb{Q}).
\end{align*}
Let us set 
\begin{align*}
& \mathrm{t}_a=\left[\begin{array}{cc}1 & a\\ 0 & 1\end{array}\right] &&  \&&&\mathrm{d}_b=\left[\begin{array}{cc}b & 0\\ 0 & 1
\end{array}\right], && a\in\mathbb{C},\, b\in \mathbb{C}^*.
\end{align*}

For any $a\in \mathbb{C}^{*}$, the matrix $\mathrm{d}_a$ commutes with all diagonal matrices with entries in $\mathbb{Q}$; this implies, 
with the description above, that  $$\rho(\mathrm{d}_a)=(\chi(a) x,\tau(a) y)$$ for some $\chi(a)$, $\tau(a)$ in $\mathbb{C}^*$ 
(Lemma \ref{Lem:centraldiag}). This yields two group homomorphisms $\chi$, $\tau\colon \C^{*}\to \mathbb{C}^{*}$. Observe that~$\chi$ 
is an extension of~$\widetilde{\chi}$, \emph{i.e.\ }$\chi(a)=\widetilde{\chi}(a)$ for any $a\in \mathbb{Q}$.

The equality $\mathrm{d}_a\mathrm{t}_1\mathrm{d}_{a^{-1}}=\mathrm{t}_a$ implies that
\begin{align*}
&\rho(\mathrm{t}_a)=\left(x,y+{\tau(a)}\right), &&\forall a\in \C^{*}.
\end{align*}
In particular, $\tau$ extends to an (injective) field homomorphism $\C\to \C$. The group $\GL(2,\mathbb{C})$ being generated by 
$\GL(2,\mathbb{Q})$ and $\big\{\mathrm{d}_a\, \big\vert\, a\in \mathbb{C}^{*}\big\}$, one has 
\begin{align*}
&\rho\left(\left[\begin{array}{cc}a & b\\ c & d\end{array}\right]\right)=\left(x\cdot \frac{\chi(ad-bc)}{(\tau(c)y+\tau(d))^k},
\frac{\tau(a)y+\tau(b)}{\tau(c)y+\tau(d)}\right),&& \forall \left[\begin{array}{cc}a & b\\ c & d\end{array}\right]\in \GL(2,\mathbb{C}). 
\end{align*}
The map $\rho$ is injective if and only if $\chi(a^2)\not=\tau(a^k)$ for any $a\in \mathbb{C}^*\smallsetminus\{1\}$.
\end{proof}

\vspace{8mm}

\bibliographystyle{plain}
\bibliography{biblio}
\nocite{}

\end{document}